\documentclass{preprint}
\usepackage[most]{tcolorbox}
\tcbset{thmbox/.style={enhanced jigsaw, breakable, sharp corners,
  colframe=#1, colback=#1!5, boxrule=0.4pt, leftrule=2pt}}
\tcolorboxenvironment{lemma}{thmbox=blue!50!black}

\usepackage{algorithm,algpseudocode}
\usepackage{multirow}
\usepackage{varwidth}
\usepackage{tabularx,makecell}
\setcitestyle{numbers,square,comma}

\tcbset{
pbstyle/.style={enhanced, frame empty, interior empty, arc=0mm,
coltitle=black, fonttitle=\bfseries, colbacktitle=white,
borderline={0.5mm}{0mm}{},%
attach boxed title to top center={yshift=\dimexpr-\tcboxedtitleheight/2},
boxed title style={boxrule=-1pt},varwidth boxed title,},
}
\NewTcbTheorem[crefname={problem}{problems}]{Problem}{Problem}{pbstyle}{pb}

\newcommand{\vl}{{\mathbf{l}}}
\newcommand{\vp}{{\mathbf{p}}}
\newcommand{\vu}{{\mathbf{u}}}
\newcommand{\vv}{{\mathbf{v}}}
\newcommand{\vw}{{\mathbf{w}}}

\newcommand{\cG}{{\mathcal{G}}}
\newcommand{\cL}{{\mathcal{L}}}
\newcommand{\cO}{{\mathcal{O}}}
\newcommand{\cP}{{\mathcal{P}}}
\newcommand{\cS}{{\mathcal{S}}}
\newcommand{\cT}{{\mathcal{T}}}

\newcommand{\myparagraph}[1]{\vspace*{0.5em}\par\noindent\textbf{{#1}.}}

\title[Dual first-order methods for convex hull prices]{Dual first-order methods for efficient computation of convex hull prices}
\author{Sofiane Tanji}
\affil{ICTEAM/INMA, Université catholique de Louvain, Louvain-la-Neuve, Belgium}
\email{sofiane.tanji@uclouvain.be}
\author{Yassine Kamri}
\affil{ICTEAM/INMA, Université catholique de Louvain, Louvain-la-Neuve, Belgium}
\author{François Glineur}
\affil{ICTEAM/INMA \& CORE, Université catholique de Louvain, Louvain-la-Neuve, Belgium}
\author{Mehdi Madani}
\affil{N-SIDE, Louvain-la-Neuve, Belgium}

\begin{document}
\maketitle

\begin{abstract}
Convex hull pricing, used in US electricity markets and raising interest in Europe, is a pricing rule designed to handle markets with non-convexities such as startup costs and minimum up and down times.
This pricing rule aims at minimizing total lost opportunity costs, which include both actual losses and missed profit opportunities.
The optimal prices are obtained by solving a partial Lagrangian dual of the original unit-commitment mixed-integer program, where power balance constraints are dualized.
The dual objective to minimize corresponds to a sum of non-smooth convex objective functions, where each term depends only on a single generator.
The subgradient of each term can be obtained independently by solving one mixed-integer program per generator.
In this work, we benchmark a large panel of first-order methods to solve this dual convex hull pricing problem.
We test several dual methods, most of which have not previously been considered for convex hull pricing, namely subgradient methods with three different stepsize strategies, two recent parameter-free methods, an accelerated gradient method on a smoothed dual, and new variable-metric bundle level methods, including proximal, preconditioned multi-cut, and dynamic level parameter variants.
We compare these methods on two representative sets of real-world large-scale instances and complement the comparison with a (Dantzig-Wolfe) primal column generation method shown to be efficient at computing convex hull prices, for reference.
Our numerical experiments show that the preconditioned multi-cut bundle methods clearly dominate across both datasets, with an accelerated gradient method on a smoothed dual as the most competitive alternative for short time budgets.
\end{abstract}

\section{Introduction}
\subsection{Pricing in electricity markets}\label{subsec:intro-pricing}

\myparagraph{Electricity markets}
Over the past decades, the global electricity system has shifted from centrally planned, bundled operations to a regulated market-based structure.
In such markets, a system operator is in charge of balancing generation and demand while meeting the technical constraints of producers and maximizing social welfare.
This scheduling problem is known as \emph{Unit Commitment} (UC) in the power systems literature.
However, market equilibrium prices exist if and only if there is no duality gap between the primal UC problem and its Lagrangian dual in which balance conditions have been dualized, which is in general not the case in power markets.
Indeed, producers face technical constraints such as startup costs or minimum up and down times that can only be modeled with binary variables, rendering the feasible region of the optimization problem non-convex.
This is a well-known fact, and basic examples illustrating typical market equilibrium impossibilities were given, for example, in \cite{scarf1994allocation}, using an electricity generation context.
While other interesting approaches relying on marginal pricing combined with minimum income conditions for generators (or their analogue on the demand side) have been considered in the literature (see for instance \cite{fernandez2015bilevel, madani2017mip} and references therein), Convex Hull pricing also has the advantage that the computation of the allocation and market prices is decoupled.

\myparagraph{Convex Hull pricing}
\label{par:chp}
Due to the general impossibility of having uniform prices supporting a competitive equilibrium in power systems, pricing schemes in electricity markets often combine a uniform electricity price and side payments (or uplifts) to producers. These side payments are paid when no market equilibrium exists, i.e., when producers face missed profits and lack the incentive to stay in the market.

The Convex Hull (CH) pricing approach, first introduced in \cite{ring1995dispatch} and further explored in \cite{gribik2007market}, calculates uniform prices that minimize the producers' lost opportunity costs. Because these side payments represent a deviation from market equilibrium, CH pricing somehow tries to reflect in the most accurate way the non-convexities in market prices.
In recent work, short- and long-term economic properties of CH pricing were compared to those of existing pricing rules, showing multiple advantages towards adopting CH pricing; see \cite{BYERS2023351, stevens4632401some}.

It was first observed in \cite{gribik2007market} that computing CH prices is equivalent to solving the Lagrangian dual associated with the main UC problem with dualized balance constraints, or more generally the system-wide constraints.

\myparagraph{Computational hardness of computing CH prices}
As stated above, computing CH prices exactly amounts to solving the partial Lagrangian relaxation of the Unit Commitment problem.
This is in general computationally expensive as it requires maximizing a nonsmooth concave function of the dual prices for which a subgradient can be computed by solving a Mixed Integer Linear Program (MILP) for each generator; see \cref{sub:formulation-ch} for details on this approach.
Approximate convex hull prices are often calculated using the continuous relaxation of the original UC problem. This is done, for instance, by the US ISO PJM, a regional system operator \cite{interconnection2017proposed}. This continuous relaxation yields exact convex hull prices when the convex hull of each generator's feasible set is given by its continuous relaxation.
For example, the Lagrangian multipliers of the balance condition in the continuous relaxation (of the binary variables) of the UC problem are exactly the CH prices when the technical constraints are limited to minimum up/down times and constant start-up (or shut-down) costs; see for example \cite{hua2016convex} and references therein.
In short, in those setups, a linear program of the size of the original UC problem is sufficient to compute the CH prices.
When only ramp conditions are considered, it can be shown that a similar approach can be followed; see \cite{madani2018convex}.
However, when taking into account all standard technical constraints currently in use in most unit-bidding-based markets (with detailed characteristics of the generation units), simultaneously including minimum up/down times and ramp conditions, the continuous relaxation of the generators' feasible sets is not tight and the continuous relaxation of the primal problem doesn't yield exact CH prices. Computing CH prices in this realistic, more general setting is the subject of the present paper.

\subsection{Computational approaches for CH pricing}
\myparagraph{Primal approaches} First consider primal approaches where one solves the Lagrangian dual calculating CH prices by solving a convexified primal UC problem where each producer’s feasible set is replaced by its convex hull.
Extended formulations have been leveraged in such approaches, where the convex hulls are described as the projection of higher-dimensional polytopes; see \cite{knueven2022computationally,yu2020extended}.
In \cite{knueven2022computationally}, the large-scale extended formulation is tackled via a Benders decomposition, while in \cite{yu2020extended} extended formulations describing the convex hulls are tackled via an iterative process in which they are dynamically added only for the generators for which they are needed.
A refined version of this iterative procedure in which only part of the extended formulation for each generator is dynamically added -- corresponding to specific intervals during which units are on -- has been independently derived in \cite{paquetcomparison} as an application of a general Column Generation procedure applied to extended formulations described in \cite{sadykov2013column}.
Finally, \cite{andrianesis2021computation} uses a standard Dantzig-Wolfe decomposition method, which can also be viewed as leveraging extended formulations, where the underlying extended formulations consist of describing the convex hulls using their vertex representations.
Stabilized decomposition methods have also received growing attention on the primal side. \cite{zhang2025stabilised_ejor,pecci2025regularized} develop
regularized and stabilized Benders decompositions for capacity expansion and energy
system planning, addressing oscillation issues closely related to those motivating
bundle level methods. In a related direction, \cite{zhang2024stabilised_cor} proposes a dynamic oracle
adaptation strategy that adjusts stabilization parameters adaptively, a concept
analogous to the dynamic level parameter we study in \cref{sec:first-order-methods}.
While these primal-side stabilization techniques are orthogonal to the dual first-order benchmarking scope of our paper, they point to an interesting avenue for future work.

\myparagraph{Dual approaches}
It is well known that Dantzig-Wolfe decompositions applied to the primal formulation correspond to applying Kelley’s cutting plane algorithm to the Lagrangian dual directly in the dual space of market prices (see \cite{briant2008comparison} for the link between both methods).
However, Kelley's method (and therefore its primal counterparts, the Dantzig-Wolfe decompositions) suffers from instability in the sense that near-optimal iterates can lead to the following iterates being much further away from the optimum, even on simple problems (see e.g., \cite[Example 3.3.1]{nesterovIntroductoryLecturesConvex2004}). This motivated bundle methods which are cutting-plane algorithms with stabilization techniques; see \cite{frangioni2020standard, Lemaréchal1995new} for an extensive discussion and review of the literature on such methods.
In \cite{stevens2022application}, the authors solve the CH pricing problem via a Bundle Level Method and prove empirically that this approach is highly competitive on real-world market instances. To the best of our knowledge, apart from the subgradient method of \cite{ito2013convex}, the existing literature does not report on other dual methods applied to CH pricing.

\vskip -2\baselineskip plus -1fil
\subsection{Contributions of the paper}
We address a gap in the literature that is also relevant to practitioners: beyond the classic subgradient and level methods and despite their promising results, several other relevant first-order dual methods from the (non-smooth) convex optimization literature have not been examined. Our comparative numerical experiments of these other methods show that the \emph{proximal} level method, a variant of the level method introduced by Lemaréchal, Nemirovskii, and Nesterov in their work on bundle method variants, is among the most competitive and outperforms the classic level method. Advanced subgradient methods, with appropriate step length strategies, and an accelerated gradient method on a smoothed dual also rank among the most competitive approaches.

More generally, we propose, compare and tune a wide range of dual first-order optimization methods for computing CH prices.

This helps us identify dual methods that outperform existing approaches in the literature.
We offer a self-contained presentation of the decision-making problem at hand and provide economic interpretations whenever possible, hoping to reach a wide audience in both the power systems and optimization communities.
Furthermore, we provide an easy-to-use Julia toolbox, \texttt{ConvexHullPricing.jl},\footnote{Available here: \url{https://github.com/SofianeTanji/ConvexHullPricing.jl}.} to solve the convex hull pricing problem with a broad spectrum of methods.
The toolbox includes functions for running benchmarks, and also provides helper functions to quickly compute values of interest from the optimization results, making post-optimization tasks simpler for practitioners.

The remainder of the paper is structured as follows.
In \cref{sec:formulations}, we review the derivation of primal and dual formulations for computing CH prices.
In \cref{sec:first-order-methods}, we introduce and describe a set of eight first-order methods to solve the dual formulation.
In \cref{sec:experiments}, we conduct extensive numerical experiments to compare all presented methods on real-world instances.
In \cref{sec:conclusion}, we conclude with our recommendations among all presented dual methods.

\section{Notations and CH pricing formulations}
\label{sec:formulations}
In electricity markets, the market operator ensures that total generation meets the demand under a market clearing constraint over a given time horizon, while each participant aims to maximize their profit.
This decision-making problem, the unit commitment problem, is described in \cref{sub:unit-commitment}.
In \cref{sub:formulation-ch}, we present the CH pricing problem as a Lagrangian dual problem of the UC problem where the power balance constraints have been dualized.
\subsection{The UC problem}
\label{sub:unit-commitment}
For the sake of simplicity, we write a general UC problem, see \cite{knueven2020mixed} for more information on formulations of unit commitment problems. Let $\cG$ denote the set of generators, $\cT = \{1,\dots,T\}$ the set of time periods, $L_t$ the inelastic demand of the market consumer at time $t$ and $l_t$ the demand met by the market at time $t$, which must be inferior or equal to $L_t$. As demand is treated as inelastic, we should also consider how non-served energy penalizes total welfare: let $C_{VOLL}$ be the unit cost of non-served energy (or value of lost load).  The state of each generator at time $t$ is defined using four different variables $(p^g_t,u^g_t, v^g_t, w^g_t)$ where $p^g_t$ is the power output, $u^g_t$ its binary commitment decision and $v^g_t$, $w^g_t$ are respectively the start-up and shutdown binary decisions. The UC decision-making problem is described below:
\begin{Problem}{Unit Commitment}{unit-commitment}
\begin{align}
	\min_{(\vp,\vu, \vv, \vl)} \notag \\ F(\vp,\vu, \vv, \vl) \coloneqq  & \sum_{\substack{t \in \cT \\ g \in \cG}}  C(p^g_t, u^g_t,v^g_t) + \sum_{t \in \cT} C_{VOLL} (L_t - l_t) \label{obj:uc}         \\
	                               \text{subject to: }            & \sum_{g \in \cG} p^g_t = l_t \; \forall t \in \cT, \label{constr:uc-balance}                           \\
	                                          & (p^g,u^g, v^g, w^g)\in \cP^g \ \forall g \in \cG, \\
	                                          & 0 \leq l_t \leq L_t \; \forall t \in \cT.
\end{align}
\end{Problem}
$C^g(p^g_t, u^g_t,v^g_t)$ denotes the cost of a generator $g$ at time $t$ given its state. This quantity takes into account three types of costs: the constant marginal cost $C_{M}^g$ which is the cost to produce one additional unit of energy, the start-up cost $C_{F}^g$ which is the fixed start-up cost whenever a generator starts producing and the no load cost $C_{NL}^g$, i.e., the fixed cost of a generator when it is producing. It is computed as follows:
\begin{equation}
	C(p^g_t, u^g_t,v^g_t) = C_{M}^g p^g_t + C_{NL}^g u^g_t + C_{F}^g v^g_t.
\end{equation}
Finally, $\cP^g$ denotes the feasible region for the integral formulation of the generators.

To define it, we first introduce technical constants for each generator describing the minimum $\underline{P}^g$ and maximum running capacity $\overline{P}^g$, the minimum uptime $UT^g$ and downtime $DT^g$, the startup and shutdown levels $SU^g$ and $SD^g$ (maximum power output at the start or shutdown of a generator $g$) as well as ramping constraints: ramp-up limit $RU^g$ and ramp-down limit $RD^g$ (maximum allowable increase/decrease in power output between two consecutive time steps). The constraint $(p^g,u^g, v^g, w^g)\in \cP^g $ can now be explicitly defined as:
\begin{align}
	\begin{cases}
	u_t^g - u^g_{t-1} = v_t^g - w_t^g                                      & \forall t \in \mathcal{T}, \\
	\sum_{i=t-UT^g+1}^{t} v^g_i \leq u^g_t                                   & \forall t \in [UT^g, T],   \\
	\sum^t_{i=t-DT^g+1} w_i^g \leq 1 - u^g_t                                 & \forall t \in [DT^g, T],   \\
	\underline{P}^g u^g_t \leq p^g_t \leq \overline{P}^g u_t^g & \forall t \in \mathcal{T}, \\
	{p}^g_t - p^g_{ t-1} \leq RU^g u^g_{ t-1} + SU^g v^g_t             & \forall t \in \mathcal{T}, \\
	{p}^g_{ t-1} - p^g_t \leq RD^g u^g_t + SD^g w^g_t                  & \forall t \in \mathcal{T}, \\
	u^g_t, v^g_t, w^g_t \in \{0, 1\}                                                     & \forall t \in \mathcal{T}
	\end{cases}
\end{align}
\subsection{Formulations for CH pricing}
\label{sub:formulation-ch}
Convex hull prices are usually defined as the prices minimizing uplifts (or side payments; see Paragraph \ref{par:chp}). In this work, we use instead the following two alternative definitions, proposed and shown to be equivalent in {\cite{Geoffrion_1974}}.
\begin{lemma} \label{def:ch}
    The following definitions are equivalent:
    \begin{enumerate}
        \item Convex hull prices are the optimal variables of the Lagrangian dual associated with the UC Problem in which one dualizes the balance condition \eqref{constr:uc-balance}.
        \item Convex hull prices are the dual variables $\pi^{CH}$ associated with the power balance constraint \eqref{constr:uc-balance} of \Cref{pb:unit-commitment} where sets $\cP^g$ have been replaced by their convex hull $\text{conv}(\cP^g)$.
    \end{enumerate}
\end{lemma}
{This equivalence is a well-established result in integer programming duality. In the context of convex hull pricing, \cite{hua2016convex} generalized this relationship to accommodate specific non-linear objective functions.}
We consider Definition 1 of \cref{def:ch}. We only relax the power balance constraint \eqref{constr:uc-balance} as it is the only one binding all generators.
This yields the following Lagrangian function of Lagrange multipliers $\pi_t$ corresponding to the electricity price at time period $t$:
\begin{alignat}{2}
    \label{first-lagr}
    \cL(\pi) = \min_{\vp, \vu, \vv, \vw, \vl} F(\vp,\vu, \vv, \vl) - \sum_{t \in \cT} \pi_t \left(\sum_{g \in \cG} p^g_t - l_t\right) \\
    \text{subject to: }
    (p^g,u^g, v^g, w^g)\in \cP^g \quad \forall g \in \cG, \notag \\
    0 \leq l_t \leq L_t \quad \forall t \in \cT. \notag
\end{alignat}
CH prices $\pi^{CH}$ are the maximizers of $\cL(\pi)$.
In practice, one often considers a set of admissible prices $Q$ which corresponds to market price range restrictions.
After rearranging terms in \eqref{first-lagr}, we obtain the following Lagrangian relaxation problem, (partial) dual of UC:
\begin{equation}
	\label{pb:dual-formulation}
	\pi^{*} = \arg\max_{\pi \in Q} \left[\cL(\pi) = \cL_0(\pi) + \sum_{g \in \cG} \cL_g(\pi)\right]
\end{equation}
where:
\begin{alignat}{2} \label{pb:L0}
	\cL_0(\pi) =
	\min_{\vl} \sum_{t \in \cT} [C_{VOLL} (L_t - l_t) + \pi_t l_t]                   \\
	\text{subject to: }
	          0 \leq l_t \leq L_t \quad \forall t \in \cT \notag
\end{alignat}
and, for each $g \in \cG$,
\begin{alignat}{2} \label{pb:Lg}
	\cL_g(\pi) =
	\min_{(p^g,u^g, v^g, w^g)} \sum_{t \in \cT} [C(p^g_t, u^g_t,v^g_t) - \pi_t p^g_t] \\
	\text{subject to: }
	(p^g, u^g, v^g, w^g)\in \cP^g. \notag
\end{alignat}

\noindent This formulation is separable, meaning we can consider the pricing problem one generator at a time, allowing considerable computational benefits.

A classical result (the proof can, e.g., be found in \cite[Corollary 8.3]{conforti2014integer}) is that a subgradient of $\cL_g(\pi)$, the optimal value of problem \eqref{pb:Lg} for a given $\pi$, is given by an optimal solution $-p^g_*(\pi)$ for $\cL_g(\pi)$ since the objective is an affine function of $\pi$. Similarly, a subgradient of $\cL_0(\pi)$, the optimal value of problem \eqref{pb:L0} for a given $\pi$, is given by an optimal solution $\vl_*(\pi)$. Hence, we can compute a subgradient of $\cL$ as follows:
\begin{lemma}[First-order oracle]
\label{lem:oracle}
	Dual function $\cL$ is nonsmooth, concave and piece-wise linear. Moreover, for a given $\pi$, one subgradient is given by
	\begin{equation}
		\left[ \vl^*(\pi) - \sum_{g \in \cG} p_*^g(\pi)\right] \in \left[\hat\partial \cL_0(\pi) + \sum_{g \in \cG}\hat\partial \cL_g(\pi)\right] \subseteq \hat\partial \cL (\pi)
	\end{equation}
	where $\vl^*(\pi)$ and $\vp_*^g(\pi)$ are some optimal solutions of respectively \eqref{pb:L0} and \eqref{pb:Lg}.
\end{lemma}

\noindent The minimization problem in $\cL_0(\pi)$ can be solved in closed form.
\begin{lemma}
	We have:
    \begin{align} \label{eq:closed-L0}
    l^*_t(\pi) & = L_t \times \mathbf{1}_{\{\pi_t < C_{VOLL}\}} \quad \forall t \in \cT \\  \label{eq:closed-L0-2}
    \cL_0(\pi) &= \sum_{t \in \cT} L_t \left[C_{VOLL} - \max\{0, C_{VOLL} - \pi_t\}\right]
	\end{align}
\end{lemma}
\begin{proof}
	Eq.~\eqref{eq:closed-L0} has a simple interpretation: to solve \eqref{pb:L0}, one consumes as much as possible if $C_{VOLL} > \pi_t$ and nothing if $C_{VOLL} < \pi_t$, while any value $l_t \in [0; L_t]$ is optimal when $C_{VOLL} = \pi_t$.
\end{proof} The constant term $L_t C_{VOLL}$ can be removed from \eqref{eq:closed-L0-2} as well as from \eqref{obj:uc} and \eqref{pb:L0} as far as finding optimal solutions is concerned.

The dual CH pricing problem \eqref{pb:dual-formulation} is a concave maximization problem, but solving it is equivalent to minimizing the opposite objective, a nonsmooth convex function. Because most methods are originally presented as nonsmooth convex minimization algorithms, we will from now on adopt the vocabulary of nonsmooth convex minimization by considering $\bar{\cL}(\pi) = - {\cL}(\pi)$ with subgradient $[\sum_{g \in \cG} \vp_*^{g}(\pi)] - \vl^*(\pi)$.
As recalled in \cref{subsec:intro-pricing}, convex hull prices minimize the total uplift, that is the sum over participants of lost opportunity costs.
For a fixed primal schedule, this uplift equals the gap between the primal (MILP) and dual objectives, and is minimized exactly when the prices solve the Lagrangian dual.
Hence the suboptimality gap ${\cL}^* - {\cL}(\pi)$ returned by any benchmarked method is precisely the additional uplift incurred relative to that minimum, independently of the chosen primal schedule.

\section{First-order methods}
\label{sec:first-order-methods}
A myriad of first-order methods exist to solve nonsmooth convex optimization problems with guaranteed convergence. Using any of those we can compute  CH prices using the dual formulation \eqref{pb:dual-formulation}, and perform at each iteration a first-order oracle call (see Lemma~\ref{lem:oracle}), which amounts to  solving $|\mathcal{G}|$ (small) MILPs in \eqref{pb:Lg}, e.g., with an off-the-shelf solver.

We leave the discussion on the stopping criteria to \cref{subsec:benchmark}. We follow the common approach (see, e.g., \cite{stevens2022application}) of setting the set of admissible prices $Q$ as $Q = [\pi_{\min}, \pi_{\max}]^T$. Most methods presented below are new in the context of computing CH prices, apart from the bundle level method (\cref{alg:BLM}) proposed in \cite{stevens2022application} and the subgradient method with vanishing step lengths (\cref{alg:SM_s_l}) proposed in \cite{ito2013convex}.
\subsection{Subgradient-based methods}
Projected subgradient methods \cite{boyd2003subgradient} are the simplest iterative algorithms for solving nonsmooth convex minimization problems, with the following algorithmic description:
\begin{algorithm}[H]
	\caption{Subgradient method (SUBG, vanishing step lengths)}
	\begin{algorithmic}
		\smallskip
		\State \textbf{Parameters:} initial step length $\eta$
		\smallskip
		\State \textbf{Inputs:} initial iterate $\pi^1 \in Q$
		\smallskip
		\Repeat{ for $k=1, 2, \dots$}
		\begin{enumerate}
			\item
			      Select a subgradient $g^k\in \partial \bar{\cL}(\pi^k)$.
			\item
			      Compute $\pi^{k+1}=P_Q\left(\pi^k- \tfrac{\eta}{k} \frac{g^k}{\|g^k\|} \right)$.
		\end{enumerate}
        \Until{stopping criterion is satisfied.}
		\smallskip
		\State \textbf{Output:} best iterate $\pi_{\arg\min_k \{\bar{\cL}(\pi_k)\} }$
	\end{algorithmic}
	\label{alg:SM_s_l}
\end{algorithm}
We present a version using normalized subgradients and therefore call the stepsize a step length. This presents the theoretical advantage of not requiring knowledge of an upper bound on the subgradient norm.
Note that one could also use a step length schedule decreasing in $\tfrac{\eta}{\sqrt{k}}$, with similar results.
The projected subgradient method (\cref{alg:SM_s_l}) is one of the earliest approaches to compute CH prices; see \cite{ito2013convex}.
The method is however notoriously slow and hard to tune as its performance is significantly influenced by the choice of both the initial iterate $\pi^1$ and the stepsize sequence, i.e., initial step length $\eta$.

With regard to the initial iterate, we propose a warm-start strategy in \cref{subsec:warm-start}. With regard to the stepsize sequence, B. Polyak suggested in 1963 \cite{polyak1963gradient} a particular stepsize for gradient-type methods (now called the Polyak stepsize) whenever the value of the optimum $\bar{\mathcal{L}}(\pi^*)$ is known:
\begin{equation}
    h_k = \frac{\bar{\mathcal{L}}(\pi^k) - \bar{\mathcal{L}}(\pi^*)}{\|g^k\|^2}
\end{equation}
As the value $\bar{\mathcal{L}}(\pi^*)$ is not known in practice, an existing idea (see \cite{boyd2003subgradient}) is to approximate it using the lowest objective value so far $\bar{\cL}^k_{best} = \min_{i \leq k} \{\bar{\cL}^i\}$ and an estimate of the gap $\frac{\alpha}{k}$, which yields the following:
\begin{algorithm}[H]
	\caption{Subgradient method (SUBG-EP, estimated Polyak stepsizes)}
	\begin{algorithmic}
		\smallskip
		\State \textbf{Parameters:} initial suboptimality estimate $\alpha > 0$
		\smallskip
		\State \textbf{Inputs:} initial iterate $\pi^1 \in Q$
		\smallskip
		\Repeat{ for $k=1, 2, \ldots$}
		\begin{enumerate}
			\item
			      Select a subgradient $g^k\in \partial \bar{\cL}(\pi^k)$.
			\item
			      Get $\pi^{k+1}=P_Q\left(\pi^k- \left(\bar{\cL}^k - \left(\bar{\cL}^k_{best} - \tfrac{\alpha}{k}\right) \right) \frac{g^k}{\|g^k\|^2} \right)$.
		\end{enumerate}
        \Until{stopping criterion is satisfied.}
		\smallskip
		\State \textbf{Output:} best iterate $\pi_{\arg\min_k \{\bar{\cL}(\pi_k)\} }$
	\end{algorithmic}
\end{algorithm}
We also investigate the efficiency of a recently proposed \cite{zamani2023exact} optimal variant of the subgradient method with a specific linearly decaying step length schedule that achieves the best possible last-iterate accuracy after a given fixed number of iterations, presented below:
\begin{algorithm}[H]
	\caption{Last-iterate optimal subgradient method (SUBG-L)}
	\begin{algorithmic}
		\smallskip
		\State \textbf{Parameters:} Upper bound $R$ such that $\|\pi^1 - \pi^*\| \leq R$
		\smallskip
		\State \textbf{Inputs:} initial iterate $\pi^1 \in Q$, number of iterations $N$
		\smallskip
		\State For $k=1, 2, \ldots, N$ perform the following steps:
		\begin{enumerate}
			\item
			      Select a subgradient $g^k\in \partial \bar{\cL}(\pi^k)$.
			\item
			      Compute $\pi^{k+1}=P_Q\left(\pi^k- \tfrac{R(N+1-k)}{\sqrt{(N+1)^3}} \frac{g^k}{\|g^k\|} \right)$.
		\end{enumerate}
		\smallskip
		\State \textbf{Output:} last iterate $\pi_{N+1}$
	\end{algorithmic}
\end{algorithm}
\subsection{Parameter-free methods}
Parameter-free methods are methods that do not require knowledge of the problem parameters to run, such as an upper bound on the subgradient norm.
\subsubsection{Bundle (level) methods}
Bundle level methods \cite{Lemaréchal1995new} are arguably among the first \emph{parameter-free} algorithms for nonsmooth convex optimization. The idea of bundle methods is to aggregate subgradient information in a bundle to maintain a polyhedral approximation of the objective function. In the CH pricing context, $\bar{\mathcal L} (\pi)$ is itself piecewise linear thus the approximation is exact for each facet on which a subgradient has been computed. This bundle approach changes considerably from previously considered algorithms where at each iteration, the current iterate is updated using only the subgradient at the previous point. The level term refers to how new iterates are computed and added to the bundle.

We present two bundle method variants: the Bundle Level Method (BLM), introduced in \cite{Lemaréchal1995new} and first used in the context of CH pricing in \cite{stevens2022application} and the Bundle Proximal Level method (BPLM), also introduced in \cite{Lemaréchal1995new} but never used in the context of CH pricing.
\begin{algorithm}[H]
	\caption{Bundle level method (BLM)}
	\begin{algorithmic}
		\smallskip
		\State \textbf{Parameters:} level set parameter $\alpha \in (0,1)$.
		\smallskip
		\State \textbf{Inputs:} initial iterate $\pi^1 \in Q$
		\smallskip
		\State Initialize $UB = +\infty, LB = -\infty$.
		\smallskip
		\Repeat{ for $k=1,2,\ldots$}
		\begin{enumerate}
			\item Compute first-order oracle $(\bar{\cL}(\pi^k), g^k\in \partial \bar{\cL}(\pi^k))$.
			\item Update upper bound $UB = \min \{UB, \bar{\cL}(\pi^k)\}$.
			\item Update lower bound $LB = \min \{ \, t \text{ s.t. } \pi \in Q \text{ and }   \bar{\cL}(\pi^i) + \langle g^i, \pi - \pi^i \rangle \leq t, \; \forall i \leq k\}$.
			\item Update level set $\cS = LB + \alpha (UB - LB)$.
			\item Update iterate $\pi^{k+1} = \arg\min \{ \, \|\pi - \pi^k\| \text{ s.t. } \pi \in Q \text{ and }   \bar{\cL}(\pi^i) + \langle g^i, \pi - \pi^i \rangle \leq \cS, \; \forall i \leq k\}$.
		\end{enumerate}
        \Until{stopping criterion is satisfied.}
		\smallskip
		\State \textbf{Output:} best iterate $\pi_{\arg\min_k \{\bar{\cL}(\pi_k)\} }$
	\end{algorithmic}
	\label{alg:BLM}
\end{algorithm}
As in the original description of the Bundle Proximal Level Method \cite{Lemaréchal1995new}, the best iterate so far is projected on the level set at each iteration (see line 7 of~\cref{alg:BPLM}).
\begin{algorithm}[H]
	\caption{Bundle proximal level method (BPLM)}\label{alg:BPLM}
	\begin{algorithmic}
		\smallskip
		\State \textbf{Parameters:} level set parameter $\alpha \in (0,1)$.
		\smallskip
		\State \textbf{Inputs:} initial iterate $\pi^1 \in Q$
		\smallskip
		\State Initialize $UB = +\infty, LB = -\infty, \Delta = +\infty, \cS' = +\infty$.
		\smallskip
		\Repeat{ for $k=1, 2, \ldots$}
		\begin{enumerate}
			\item Compute first-order oracle $(\bar{\cL}(\pi^k), g^k\in \partial \bar{\cL}(\pi^k))$.
			\item Update upper bound $UB = \min \{UB, \bar{\cL}(\pi^k)\}$.
			\item Update lower bound $LB = \min \{ \, t \text{ s.t. } \pi \in Q \text{ and }   \bar{\cL}(\pi^i) + \langle g^i, \pi - \pi^i \rangle \leq t, \; \forall i \leq k\}$.
			\item Update level set $\cS = LB + \alpha (UB - LB)$.
            \item $
            \cS' =
            \begin{cases}
            \min\{\cS, \cS'\} & \text{if } UB - LB \geq (1-\alpha)\Delta,\\[1mm]
            \cS & \text{otherwise,}
            \end{cases}
            $
            \item $
            \Delta =
            \begin{cases}
            \Delta & \text{if } UB - LB \geq (1-\alpha)\Delta,\\[1mm]
            UB - LB & \text{otherwise.}
            \end{cases}
            $
			\item Update iterate $\pi^{k+1} = \arg\min \{ \, \|\pi - \hat{\pi}^{k}\| \text{ s.t. } \pi \in Q \text{ and }   \bar{\cL}(\pi^i) + \langle g^i, \pi - \pi^i \rangle \leq \cS', \; \forall i \leq k\}$, where $\hat{\pi}^{k} \in \arg\min_{i \leq k} \bar{\cL}(\pi^i)$.
		\end{enumerate}
        \Until{stopping criterion is satisfied.}
		\smallskip
		\State \textbf{Output:} best iterate $\pi_{\arg\min_k \{\bar{\cL}(\pi_k)\} }$
	\end{algorithmic}
\end{algorithm}
Both bundle level algorithms make at each iteration one oracle call, then solve an LP to update the polyhedral model's lower bound LB (line 3 in both \cref{alg:BLM,alg:BPLM}), then solve a QP to project onto the level set
(line 5 in \cref{alg:BLM} and line 7 in \cref{alg:BPLM}). The difference lies in the level set computation: in BPLM (\cref{alg:BPLM}), $\cS'$ cannot increase until the gap $\Delta$ decreases substantially (line 5), whereas in BLM (\cref{alg:BLM}), it is updated at every iteration.

An interesting advantage of both bundle methods is that the gap $\Delta = UB - LB$ provides us with an estimate of the current accuracy. This will be leveraged in \cref{subsec:optimal-sol} to compute near-optimal convex hull prices (\cref{fig:blm-geometry}).

\begin{figure}[htbp]
\centering
\includegraphics[width=\linewidth]{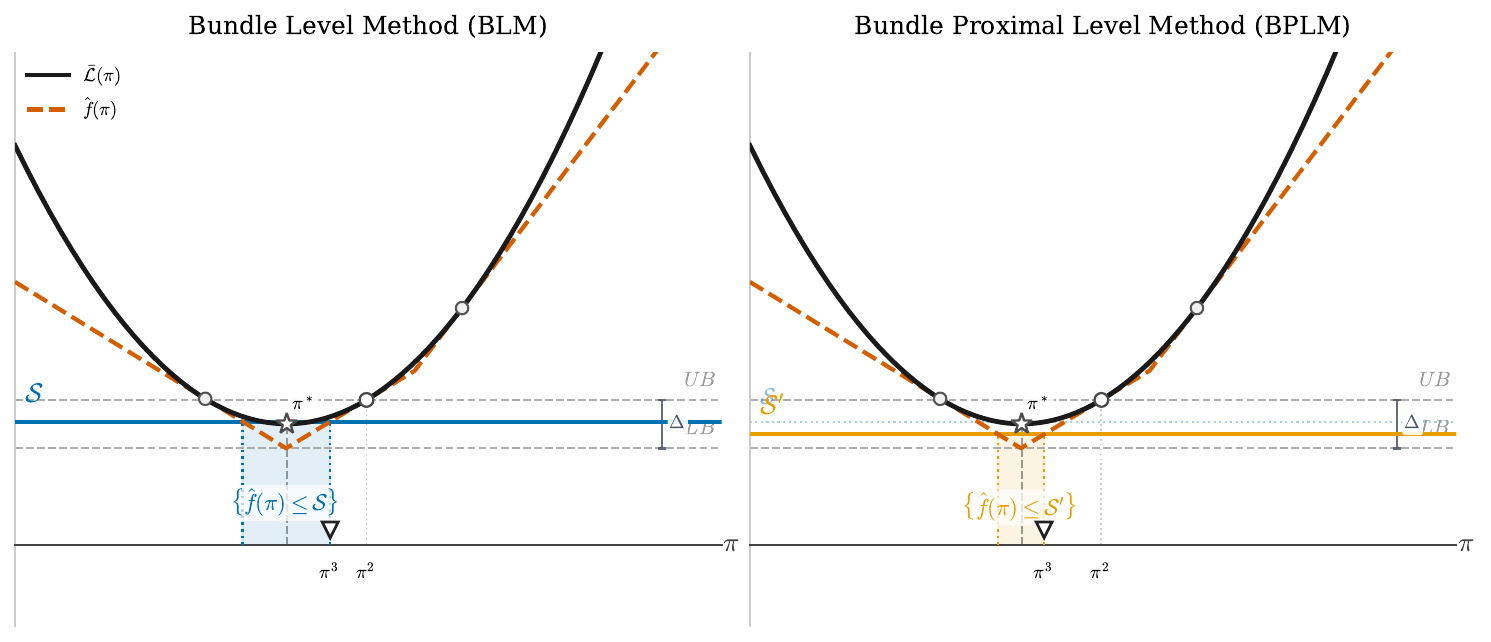}
\caption{{Comparison of BLM (left) and BPLM (right) on minimizing $f(\pi)=0.3(\pi-2.5)^2+1$, with $\pi^0=1.75$ and $\alpha=0.55$.
Both panels depict iteration $k=2$, the first where both methods differ. Because BPLM uses the tighter level set $\cS'$, it produces an Update closer to $\pi^*$: $|\pi^{3}_{\mathrm{BPLM}}-\pi^*|=0.207$ versus $0.400$ for BLM.}}
\label{fig:blm-geometry}
\end{figure}

\myparagraph{Multi-cut bundle methods}
The bundle methods considered above are single-cut methods: at each oracle call, the $|\cG|$ subproblems in \eqref{pb:Lg} are aggregated into a single cut by summing subgradients of the individual functions $-\cL_g$, thereby obtaining one subgradient of $\bar{\cL}$.
Since $\bar{\cL}(\pi)=-\cL_0(\pi)-\sum_{g\in\cG}\cL_g(\pi)$ is separable, an alternative is to maintain a distinct polyhedral model for each generator, the term $-\cL_0$ being shared equally among the generators.
Each oracle call then adds $|\cG|$ cuts, one per generator, resulting in a tighter lower approximation for the same oracle information.
We denote the corresponding variants by MC-BLM and MC-BPLM.
MC-BLM is mentioned in \cite{stevens2022application}, where it is reported to improve performance.

The resulting lower-bounding model minimizes $\sum_g t_g$ subject, for each generator $g\in\cG$ and each past iterate $i$, to $\phi_g(\pi^i)+\langle g_g^i,\pi-\pi^i\rangle \le t_g$, where $\phi_g = -\cL_g - \cL_0/|\cG|$ and $g_g^i$ is a subgradient of $\phi_g$ at $\pi^i$.
The associated level-set projection problem is defined analogously, with one slack variable per generator.

\myparagraph{Variable-metric bundle methods}
In BLM and BPLM, the projection step is a Euclidean projection onto the level set.
We also consider a diagonal variable metric in this step, replacing $\frac12\|\pi-\pi^k\|^2$ by
\begin{equation*}
\frac12\sum_j w_j(\pi_j-\pi_j^k)^2,
\end{equation*}
with $\hat{\pi}^k$ in place of $\pi^k$ in BPLM.
The weights are updated at every iteration from past subgradient information using a running root-mean-square estimate of the subgradient $g^k$, and they are clipped:
\begin{gather*}
r_j^k = \sqrt{\beta (r_j^{k-1})^2 + (1-\beta)(g_j^k)^2}, \\
w_j = \min\left\{\max\left\{(r_j^k + \kappa)^{-2}, 10^{-4}\right\}, 10^{4}\right\},
\end{gather*}
with $r_j^0 = 1$, $\beta = 0.9$ and $\kappa = 10^{-8}$.
This yields a variable-metric variant of the projection step.
As this can be interpreted as diagonal preconditioning of
the projection problem, we denote the corresponding variants by PC-BLM and PC-BPLM.

\myparagraph{Dynamic level parameter (DLM)}
\cite{zhang2024stabilised_cor} proposes adjusting the level parameter $\alpha$ dynamically rather than keeping it fixed throughout the run.
The basic idea is to reduce $\alpha$, thereby tightening the level set, when the gap $\Delta = UB - LB$ decreases sufficiently, and to increase it otherwise.
This is similar in spirit to a trust-region mechanism, in which the admissible region is enlarged or contracted depending on the observed progress.
We implement this variant, denoted DLM, and include it in our benchmark to assess such an adaptive strategy in the CH pricing context.

\subsubsection{Recent parameter-free methods}
Parameter-free methods have regained significant attention during the last few years and multiple new methods have been proposed, in particular in the context of optimization for machine learning; see \cite{nesterov2015universal, orabona2016coin,cutkosky2017online, roulet2019sharpness, defazio2023dadapt, khaled2023dowg, li2023simple}. We present two promising methods among these, D-Adaptation \cite{defazio2023dadapt} and DowG \cite{khaled2023dowg} and adapt them to our CH pricing problem.

\begin{algorithm}[H]
	\caption{D-adaptation (DA)}
	\begin{algorithmic}
		\smallskip
		\State \textbf{Parameters:} initial lower bound $D_1$ on $\|\pi^1 - \pi^*\|$
		\smallskip
		\State \textbf{Inputs:} initial iterate $\pi^1 \in Q$
		\smallskip
		\State Initialize $s^1 = 0, \gamma^1 = \|g^1\|^{-1}$.
		\smallskip
		\Repeat{ for $k=1, 2, \ldots$}
		\begin{enumerate}
			\item Select a subgradient $g^k\in \partial \bar{\cL}(\pi^k)$.
			\item Compute $s^{k + 1} = s^k + D_k g^k$.
			\item Compute $\gamma^{k + 1} = (\sum_{i = 1}^k \|g^i\|^2)^{-\tfrac{1}{2}}$.
			\item $D_{k+1} = \max \left( D_k, \frac{\gamma^{k + 1}\|s^{k+1}\|^2 - \sum_{i \leq k} \gamma^i D_i^2 \|g^i\|^2}{2 \|s^{k+1}\|}\right)$.
			\item Compute $\pi^{k+1}=P_Q\left(\pi^k- \gamma^{k + 1} s^{k + 1}\right)$.
		\end{enumerate}
        \Until{stopping criterion is satisfied.}
		\smallskip
		\State \textbf{Output:} best iterate $\pi_{\arg\min_k \{\bar{\cL}(\pi_k)\} }$
	\end{algorithmic}
\end{algorithm}
\begin{algorithm}[H]
	\caption{Distance over weighted gradients (DoWG)}
	\begin{algorithmic}
		\smallskip
		\State \textbf{Parameters:} initial lower bound $d_1$ on $\|\pi^1 - \pi^*\|$
		\smallskip
		\State \textbf{Inputs:} initial iterate $\pi^1 \in Q$
		\smallskip
		\State Initialize $v^0 = 0$.
		\smallskip
		\Repeat{ for $k=1, 2, \ldots$}
		\begin{enumerate}
			\item Select a subgradient $g^k\in \partial \bar{\cL}(\pi^k)$.
			\item Compute distance estimator $d_{k+1} = \max (d^{k} , \|\pi^k - \pi^1\|)$.
			\item Compute weighted gradient sum $v^{k} = v^{k-1} + d_{k+1}^2 \|g^k\|^2$.
			\item Compute step size $\eta_k = \tfrac{d_{k+1}^2}{\sqrt{v^k}}$.
			\item Compute step $\pi^{k+1}=P_Q\left(\pi^k- \eta_k g^k\right)$.
		\end{enumerate}
        \Until{stopping criterion is satisfied.}
		\smallskip
		\State \textbf{Output:} best iterate $\pi_{\arg\min_k \{\bar{\cL}(\pi_k)\} }$
	\end{algorithmic}
\end{algorithm}

Both methods try to circumvent the issue of not knowing the distance to the optimal solution $R$ by using an initial distance estimate. This estimate is then updated at each iteration to approximate the true distance. This approach helps in selecting the stepsize adaptively and allows convergence without knowing any of the problem parameters.
\subsection{Smoothing-based methods}
First-order methods for convex nonsmooth minimization have a
convergence rate of order $\cO(\varepsilon^{-2})$, where $\varepsilon$ is the desired accuracy for the approximate solution, and the order of this complexity is optimal. On the other hand, simple convex smooth minimization methods have convergence rates of order $\cO(\varepsilon^{-1})$ or $\cO(\varepsilon^{-0.5})$ for the fastest algorithms. In order to bypass the slow rate inherent in nonsmooth minimization, we use a smooth approximation of the original problem \eqref{pb:dual-formulation} using Nesterov's smoothing technique \cite{nesterov2005smooth} and apply a fast projected gradient method \cite{beck2009fast} to solve the given approximation.

We recall the original dual objective $\cL(\pi) = \cL_0(\pi) + \sum_{g \in \cG} \cL_g(\pi)$
with $\cL_0$ and $\cL_g$ described in \eqref{pb:L0} and \eqref{pb:Lg}.
We propose to smooth both $\cL_0$ and $\cL_g$ as both are nondifferentiable and have an explicit min-structure. Given reference points $\vl^{r}$ and $(\vp^g,\vu^g, \vv^g, \vw^g)^{r}$ and a smoothing parameter $\varsigma > 0$, we introduce
\begin{alignat}{2} \label{smooth-L0}
    \cL_{0,\varsigma}(\pi) &= \min_{\vl} \sum_{t \in \cT} \big[ C_{VOLL} (L_t - l_t) + \pi_t l_t\big]
    + \frac{\varsigma}{2} \|\vl - \vl^{r}\|^2  \notag \\
    &\quad \text{s.t.} \quad 0 \leq l_t \leq L_t, \quad \forall t \in \cT.
\end{alignat}

\vspace{-4mm}
\begin{alignat}{2} \label{smooth-Lg}
    \cL_{g,\varsigma}(\pi) &= \min_{(\vp^g,\vu^g, \vv^g, \vw^g)} \sum_{t \in \cT} \big[ C(p^g_t, u^g_t, v^g_t)
    - \pi_t p^g_t\big]  \notag \\
    &\quad + \frac{\varsigma}{2} \|(\vp^g, \vu^g , \vv^g, \vw^g) - (\vp^g, \vu^g, \vv^g, \vw^g)^{r}\|^2 \notag \\
    &\quad \text{s.t.} \quad (\vp^g, \vu^g, \vv^g, \vw^g) \in \text{conv}(\cP^g),
\end{alignat}
which are the smooth approximations of $\cL_0$ and $\cL_g$.
Since the objective of \eqref{pb:Lg} is linear, replacing $\cP^g$ by its convex hull does not change $\cL_g$, and \eqref{smooth-Lg} becomes a strongly convex problem over a polytope.
The convex hull is necessary: over the nonconvex set $\cP^g$, the minimizer need not be unique, and $\cL_{g,\varsigma}$ is in general not differentiable at prices where two commitment schedules are both optimal.
\begin{lemma}
	Given smoothing parameter $\varsigma > 0$ and reference points $\vl^{r}$ and $(\vp, \vu, \vv, \vw)^{r}$, the full smooth approximation $\tilde{\cL}_\varsigma = -\cL_{0,\varsigma} - \sum_{g \in \cG} \cL_{g,\varsigma}$ of $\bar{\cL}$ is convex and differentiable, with gradient
	\begin{equation}
		\nabla \tilde{\cL}_\varsigma (\pi) =  \sum_{g \in \cG} (p^g_{\varsigma,*}(\pi)) - l^{*}_\varsigma(\pi)
	\end{equation}
 	where $l^{*}_\varsigma(\pi)$ is the optimal solution of \eqref{smooth-L0} and $p_{\varsigma,*}^g(\pi)$ is the production part of the optimal solution of \eqref{smooth-Lg}, both solutions being unique by strong convexity of the objective functions over convex feasible sets. Moreover, the gradient $\nabla \tilde{\cL}_\varsigma$ is Lipschitz continuous with constant $\tfrac{1 + |\cG|}{\varsigma}$.
\end{lemma}
\begin{proof}
	Each of the $1 + |\cG|$ terms of $\tilde{\cL}_\varsigma$ is the maximum over a nonempty compact convex set of a function that is affine in $\pi$ and $\varsigma$-strongly concave in the primal variables. The linear map from $\pi$ to the primal space has norm $1$ in each term. By \cite[Theorem 1]{nesterov2005smooth}, each term is therefore convex and differentiable, with gradient given by its unique maximizer and $\tfrac{1}{\varsigma}$-Lipschitz continuous. Summing the terms gives the result.
\end{proof}
By this lemma, computing $\nabla \tilde{\cL}_\varsigma$ requires the production part of the solution of \eqref{smooth-Lg} for each generator, while \eqref{smooth-L0} is separable and has a closed-form solution.
We compute the solution of \eqref{smooth-Lg} by column generation on the vertex representation of $\text{conv}(\cP^g)$.
Write $x = (\vp^g, \vu^g, \vv^g, \vw^g)$ and let $q_\pi^g(x)$ be the objective of \eqref{smooth-Lg}.
Given a set $V^g \subset \cP^g$ of schedules, the restricted master problem $\min_{x \in \text{conv}(V^g)} q_\pi^g(x)$ is a convex QP in the weights of the schedules.
Its solution $\bar{x}^g$ gives the costs $d^g = \nabla q_\pi^g(\bar{x}^g)$, and the pricing problem $\min_{x \in \cP^g} (d^g)^\top x$ is a MILP of the same size as \eqref{pb:Lg}.
Let $B^g$ be a lower bound on its value, such as the best bound returned by the MILP solver, and let $\delta^g = (d^g)^\top \bar{x}^g - B^g$.
We stop when $\delta^g \leq \tau$ or after $50$ master iterations, and otherwise add the pricing solution to $V^g$ (see \cref{fig:cg-oracle}).
The columns are kept from one iteration of the fast gradient method to the next, since consecutive prices are close.
\begin{figure}[htbp]
	\centering
	\begin{tikzpicture}[
		>=stealth,
		box/.style={draw=black!35, fill=black!2, rounded corners=.08cm, minimum height=.88cm, align=center, font=\scriptsize},
		lab/.style={font=\scriptsize}]
		\node[box, minimum width=4.6cm] (master) at (-2.9,0) {\textbf{Master QP}\\combine the stored schedules $V^g$};
		\node[box, minimum width=4.6cm] (pricing) at (2.9,0) {\textbf{Pricing MILP}\\find the best schedule for costs $d^g$};
		\node[box, minimum width=3cm] (test) at (0,-2.1) {\textbf{Stopping test}\\$\delta^g \leq \tau$ ?};
		\draw[<-] (master.north) -- ++(0,.5) node[lab, above] {price $\pi$ from FGM};
		\draw[->] (master.east) -- (pricing.west) node[lab, midway, above] {$\bar{x}^g$, $d^g$};
		\draw[->] (pricing.south) |- (test.east) node[lab, pos=.25, right, align=left] {new schedule $x^{\text{new}}$,\\gap $\delta^g$};
		\draw[->] (test.west) -| (master.south) node[lab, pos=.75, left, align=right] {no: add $x^{\text{new}}$\\to $V^g$};
		\draw[->] (test.south) -- ++(0,-.7) node[lab, midway, right] {yes} node[lab, below] {production of $\bar{x}^g$, to FGM};
	\end{tikzpicture}
	\caption{Column generation for \eqref{smooth-Lg} at the price $\pi$. The master returns the best point $\bar{x}^g$ in the convex hull of the stored schedules $V^g$. The pricing MILP minimizes over $\cP^g$ the linearization of $q_\pi^g$ at $\bar{x}^g$, with costs $d^g = \nabla q_\pi^g(\bar{x}^g)$, and returns a new schedule and the gap $\delta^g = (d^g)^\top \bar{x}^g - B^g$.}
	\label{fig:cg-oracle}
\end{figure}
\begin{lemma}
	Let $\bar{x}^g$ and $\delta^g$ be returned by the column generation, let $x_*^g$ be the solution of \eqref{smooth-Lg}, and let $\bar{p}^g$ and $p_*^g$ be their production parts. Then
	\begin{equation*}
		0 \leq q_\pi^g(\bar{x}^g) - \cL_{g,\varsigma}(\pi) \leq \delta^g
		\quad \text{and} \quad
		\|\bar{p}^g - p_*^g\| \leq \sqrt{2\delta^g / \varsigma}.
	\end{equation*}
\end{lemma}
\begin{proof}
	The first inequality holds since $\bar{x}^g \in \text{conv}(\cP^g)$. A linear function attains its minimum over $\text{conv}(\cP^g)$ at a point of $\cP^g$, hence $(d^g)^\top x_*^g \geq B^g$. By convexity of $q_\pi^g$, we then have $\cL_{g,\varsigma}(\pi) = q_\pi^g(x_*^g) \geq q_\pi^g(\bar{x}^g) + (d^g)^\top(x_*^g - \bar{x}^g) \geq q_\pi^g(\bar{x}^g) - \delta^g$. Finally, $\varsigma$-strong convexity of $q_\pi^g$ and optimality of $x_*^g$ over $\text{conv}(\cP^g)$ give $\tfrac{\varsigma}{2}\|\bar{x}^g - x_*^g\|^2 \leq q_\pi^g(\bar{x}^g) - q_\pi^g(x_*^g) \leq \delta^g$, and $\|\bar{p}^g - p_*^g\| \leq \|\bar{x}^g - x_*^g\|$.
\end{proof}
The bound only uses $B^g$, so it also holds when the pricing MILPs are not solved to optimality.
With $\epsilon = \sum_{g \in \cG} \delta^g$, this yields a $(2\epsilon, 2L)$-oracle with $L = (1 + |\cG|)/\varsigma$ \cite[Section 3.1]{devolder2014first}, for which the fast gradient method converges on the smoothed objective at the rate $\cO(1/k^2)$, up to an additional accumulated error $2(k + 3)\epsilon/3$ \cite[Equation 39]{devolder2014first}.
With the tolerance $\tau = 10^{-7}$ used in our experiments and the number of iterations FGM performs in $15$ minutes, this error is below one currency unit over a whole market when every generator stops with $\delta^g \leq \tau$.

We will use as fixed reference points for the smooth approximation the optimal solutions of the continuous linear relaxation of \Cref{pb:unit-commitment}.
The fast gradient algorithm used to solve the above smooth approximation is provided in \Cref{alg:FGM}.
\begin{algorithm}[H]
	\caption{Fast (projected) gradient method (FGM)}
	\begin{algorithmic}
		\smallskip
		\State \textbf{Parameters:} initial step length $\eta$, smoothing parameter $\varsigma$
		\smallskip
		\State \textbf{Inputs:} initial iterate $\pi^1 \in Q$
		\smallskip
        \State Initialize $y^1 = \pi^1$.
		\smallskip
		\Repeat{ for $k=1, 2, \ldots$}
		\smallskip
		\begin{enumerate}
			\item Compute $\nabla \tilde{\cL}_\varsigma (y^k)$.
			\item
			      Compute $\pi^{k+1}=P_Q\left(y^k- \eta \nabla \tilde{\cL}_\varsigma (y^k) \right)$.
			\item Compute $y^{k + 1} = \pi^{k+1} + \tfrac{k - 1}{k + 2}(\pi^{k+1} - \pi^k)$.
		\end{enumerate}
        \Until{stopping criterion is satisfied.}
		\smallskip
		\State \textbf{Output:} best iterate $\pi_{\arg\min_k \{\bar{\cL}(\pi_k)\} }$
	\end{algorithmic}
	\label{alg:FGM}
\end{algorithm}
{%
We note that the values reported for FGM are computed on the original (unsmoothed) objective $\bar{\cL}(\pi)$ evaluated at the iterates generated by the smoothed problem, not on the surrogate $\tilde{\cL}_\varsigma(\pi)$.
This makes the reported objective values directly comparable with those of the other methods.

The discrepancy between the two objectives is bounded by the standard smoothing estimate, which holds for any $\pi$:
\begin{equation*}
0 \le \bar{\cL}(\pi)-\tilde{\cL}_\varsigma(\pi) \le \varsigma D,
\end{equation*}
where $D$ denotes the prox-diameter of the primal feasible region, see for example \cite[Equation 2.7]{nesterov2005smooth}.
Consequently, any minimizer of $\tilde{\cL}_\varsigma$ over $Q$ is $\varsigma D$-optimal for $\bar{\cL}$.

The reference point used in the smoothing construction acts as a primal selection mechanism: for a given $\pi$, it favors primal minimizers close to the reference point when multiple minimizers exist.
It therefore affects the gradient of the smoothed dual and, consequently, the geometry of the surrogate problem solved by FGM and the trajectory of the iterates, but it does not modify the original CH-pricing dual objective on which the final prices are evaluated.
}

\section{Numerical experiments}
\label{sec:experiments}
{
We now study the performance of the proposed methods on real-world UC instances.
We first describe the common experimental protocol and the implementation choices shared by all methods.
We then describe the main implementation choices for the multi-cut bundle variants and compare the bundle methods within this family to determine which variants are retained for the main benchmark.
We next compare these retained bundle methods with the other algorithmic families.
Finally, we report diagnostic and robustness analyses.
}
\subsection{Experimental protocol}
\myparagraph{Experimental setup}
Each method is coded in Julia (1.12.5) using JuMP (1.29.4) and uses Gurobi (13.0) with default parameters, except for the relative MIP gap tolerance which is set to $10^{-8}$ to solve subproblems.
This tolerance is uncorrelated with the significant digits of the computed convex hull prices and its small value is supported by preliminary experiments indicating better numerical stability for all benchmarked methods.
The codebase used for our experiments is distributed as an open-source Julia package \texttt{ConvexHullPricing.jl}.
All tuning and benchmark runs were performed on AMD EPYC 9655 processors at 2.6 GHz, using 8 threads and 15 GB of RAM per run.
\myparagraph{Datasets}
We run all our experiments on 24 real-world UC problem instances.
Specifically, we consider  a \emph{Belgian} dataset containing 8 real-world instances (one for each season (winter, spring, summer, autumn) and for each weekly period (weekday, weekend)) with 68 Belgian generators on a 24-hour time horizon with a 15~min time-step.
We also consider a \emph{Californian} dataset containing 16 curated instances\footnote{Among the 20 instances in the provided repository, we exclude 4 of them because they are not relevant for comparing methods, as our warm-start technique directly provides CH prices.} extracted from publicly available data\footnote{Data available at \href{https://github.com/power-grid-lib/pglib-uc}{\nolinkurl{https://github.com/power-grid-lib/pglib-uc}}.} from the California ISO.

Considering instances where the price space has dimensions of 48 and 96 (i.e., one price per period {over a 24-hour horizon}) allows us to examine how sensitive the performance results are to the dimensionality of the price space. Instances of this type, without a transmission network, were already analyzed in \cite{knueven2022computationally,stevens2022application}. The role of problem dimensionality was further emphasized in \cite{stevens2022application}, which additionally considered instances with multiple locations connected through a transmission network to investigate its impact in greater detail. Even without multiple locations, they observed that the standard Dantzig-Wolfe approach tested in \cite{andrianesis2021computation} is already adversely impacted by an increase from 12 to 96 periods, in contrast to the Level method. In this work, we focus on similar instance types with a single location but still with such a high number of periods, common to the references \cite{knueven2022computationally, stevens2022application, andrianesis2021computation}, both to focus on the breadth of the dual methods tested and to facilitate future comparisons with the primal methods discussed in \cite{andrianesis2021computation, knueven2022computationally}. Future work could explore extending the numerical experiments to instances with multiple locations connected by advanced transmission networks (\cref{tab:datasets,tab:hyperparameters,tab:alpha_vs_varsigma_fmin_new}).

We summarize these datasets in the table below:
\begin{table}[hbt]
    \centering
    \caption{Summary of the datasets.}
    \label{tab:datasets}
    \begin{tabular}{lcccc}
        \toprule
        {Dataset name} & {\# instances} & {\# generators} & {\# time periods} & {source}                      \\ \midrule
        Belgian & 8 & 68 & 96 & \cite{paquetcomparison, stevens2022application} \\
        Californian & 16 & 610 & 48 & \cite{knueven2020mixed} \\
        \bottomrule
    \end{tabular}
\end{table}

{The 48 time periods in the Californian dataset correspond to half-hour intervals
over a 24-hour horizon (i.e., 48 $\times$ 30\,min), while the Belgian dataset uses 15-minute
intervals over the same horizon (96 $\times$ 15\,min).}

\myparagraph{Computation of optimal values \texorpdfstring{$\mathcal{L}_*$}{L*}}
\label{subsec:optimal-sol}

All considered datasets originally describe unit commitment instances, and thus do not include the set of optimal solutions to our CH pricing problem.

As we want to compare the (relative) objective accuracy achieved by all proposed methods (we do not compare solutions because CH prices may not be unique), we compute for each instance a very accurate estimate of the optimal value $\mathcal{L}_*$.
On each of the 24 instances, we make a long run of the Bundle Level Method with parameter $\alpha = 0.2$ and exact oracle computations (Gurobi, the MILP solver in our experiments, is parametrized with absolute and relative MILP gap tolerance of $0$) until the absolute gap between the constructed upper and lower bounds is lower than $1\$$, corresponding approximately to machine precision in all our experiments.
These computations took up to 6 hours.
Note that this optimal value $\mathcal{L}_*$ is not provided to the methods.

\myparagraph{Hyperparameter selection}
All presented methods feature a single hyperparameter. Hence, the tuning of that single hyperparameter for each method does not unfairly benefit any approach.
Benchmarks in the next subsection are computed after tuning all methods on one instance per dataset (on \emph{Autumn Week-day} for the \emph{Belgian} dataset and on \emph{2014-09-01\_reserves\_0} for the \emph{Californian} dataset).
We emphasize that full per-instance retuning was not performed: all results reported below are based on these globally tuned parameters, kept fixed across all instances of each dataset.
To perform our hyperparameter tuning, we run each method during $5$ minutes for each value of the hyperparameter belonging to a coarse grid with 6 values, logarithmic except for the level parameters of the bundle methods, for which it is linear. We then make a second finer grid search with 4 values around the best found parameter in the coarse grid search. Results of the hyperparameter tuning are found in the table below:
\begin{table}[H]
	\centering
	\small
	\caption{List of tuned hyperparameters for each considered optimization method.
    }
	\label{tab:hyperparameters}
		\begin{tabular}{@{}rlcc@{}}
			\toprule
			\multicolumn{1}{c}{Method} & Hyperparameter            & BE & CA \\
			\midrule
			\multicolumn{4}{l}{\textbf{Subgradient methods}} \\
            SUBG & Initial step length $\eta \in [10^{-2}; 10^3]$ & $290$ & $0.40$ \\
            SUBG-EP & Initial error estimate $\alpha \in [10^{1}; 10^5]$ & $63{,}000$ & $250$ \\
			SUBG-L & Initial distance estimate $R\in [10^{-2}; 10^3]$ & $160$ & $0.40$ \\
			\addlinespace[0.2em]
			\multicolumn{4}{l}{\textbf{Bundle methods}} \\
			{BLM}  & {Level parameter $\alpha \in [0.1,0.99]$} & {$0.99$} & {$0.99$} \\
			{MC-BLM}  & {Level parameter $\alpha \in [0.1,0.99]$} & {$0.63$} & {$0.81$} \\
			{MC-BPLM} & {Level parameter $\alpha \in [0.1,0.99]$} & {$0.99$} & {$0.34$} \\
			{PC-BLM}  & {Level parameter $\alpha \in [0.1,0.99]$} & {$0.93$} & {$0.81$} \\
			{PC-BPLM} & {Level parameter $\alpha \in [0.1,0.99]$} & {$0.87$} & {$0.63$} \\
			{DLM}     & {Initial level parameter $\alpha_0 \in [0.1,0.99]$} & {$0.93$} & {$0.99$} \\
			\addlinespace[0.2em]
			\multicolumn{4}{l}{\textbf{Parameter-free methods}} \\
			DA & Initial distance estimate $D_1 \in [10^{-2}; 10^2]$ & $16$ & $0.1$ \\
			DOWG & Initial distance estimate $d_0 \in [10^{-2}; 10^2]$ & $10$ & $0.025$ \\
			\addlinespace[0.2em]
			\multicolumn{4}{l}{\textbf{Smoothing-based method}} \\
			FGM & Constant step length $\eta \in [10^{-7}; 10^{-2}]$ & $2.2 \times 10^{-5}$ & $4.0 \times 10^{-6}$ \\
			\addlinespace[0.2em]
			\multicolumn{4}{l}{\textbf{Primal method}} \\
			CG & Tolerance parameter $\varepsilon \in [10^{-9},10^{1}]$ & $10^{-6}$    & $10^{-6}$    \\
			\bottomrule
		\end{tabular}%
\end{table}

\noindent The \emph{SUBG-L} method also requires the total number of iterations as an input parameter. To make its tuning  fair with other methods, we pick the number of iterations that results in a total computation time of about the time budget of the run, that is $5$ minutes for tuning and $15$ minutes for the main benchmark. Finally, \emph{FGM} works on a smoothed variant of the original instance. We do not consider the smoothing parameter to be a hyperparameter and fix it equal to $10^{-8}$ for both datasets.
\myparagraph{Warm starts}
\label{subsec:warm-start}
For all methods, we set the first iterate to be the dual variables of the balance constraints \eqref{constr:uc-balance} in the continuous relaxation of the UC problem.
This initialization strategy comes at a cheap cost and we have found it to reduce considerably the order of magnitude of the Lagrangian function compared to other strategies such as initialization at a fixed plausible price, say $\pi_1 = [40, 40, \dots, 40]$.
This initialization is interpretable: these prices are exactly convex hull prices in the case where we only have minimum up/down times and constant start-up/shut-down costs and the classic tight $3$-bin formulation available in that case is used; see \cite{hua2016convex}.
\subsection{Stabilized implementation of multicut bundle variants}
In preliminary experiments, the direct multi-cut formulation proved numerically unstable. The main difficulty arises in the projection problem: each iteration adds one cut per generator, so after a few iterations the master problem contains a large number of nearly redundant linear inequalities.
This degrades conditioning, especially since the cuts are derived from subgradients obtained by solving MIPs and therefore are not expected to be numerically identical even when theoretically close.
PC-BLM and PC-BPLM rely on the stabilized implementation described below, while MC-BLM and MC-BPLM keep one cut per generator.

The stabilization addresses this issue by controlling the growth and redundancy of the cut system. One possibility is to aggregate or prune cuts across iterations, but this sacrifices the monotonic decrease of the lower bound.
We instead adopt a clustered multi-cut construction. At each iteration, the generator-wise cuts are computed as usual, but the generators are partitioned once into $K = 50$ groups, according to the time period at which the absolute value of the subgradient is maximal, which serves as a computationally inexpensive proxy for clustering cuts with similar orientation.
The master problem then uses one epigraph variable and one cut per cluster rather than per generator.
This preserves the multi-cut structure while keeping the projection problem numerically manageable.
\subsection{Selection of bundle variants}
We compare several bundle variants to decide which methods to retain for the main benchmark.
To isolate the effect of the main modeling choices, we compare the single-cut baseline BLM with five variants: DLM; the multicut methods MC-BLM and MC-BPLM; and the variable-metric multicut methods PC-BLM and PC-BPLM.

\begin{table}[H]
\centering
\small
\setlength{\tabcolsep}{7pt}
\renewcommand{\arraystretch}{1.15}
\caption{Performance of bundle variants relative to the single-cut baseline BLM.
For each dataset, we report the geometric-mean speedup in the time required to reach a side-payment level of
$10^{-5}$ times the mean minimum side payment of the dataset.
Larger speedup values indicate better performance.}
\label{tab:bundle-ablation}
\begin{tabularx}{\textwidth}{@{}lXXXXXX@{}}
\toprule
Dataset & BLM & DLM & \makecell{MC\\BLM} & \makecell{MC\\BPLM} & \makecell{PC\\BLM} & \makecell{PC\\BPLM} \\
\midrule
BE & $1.00\times$ & $0.31\times$ & $0.99\times$ & $1.15\times$ & $1.25\times$ & $\mathbf{1.31\times}$ \\
\midrule
CA & $1.00\times$ & $0.75\times$ & $1.17\times$ & $1.33\times$ & $1.90\times$ & $\mathbf{2.35\times}$ \\
\bottomrule
\end{tabularx}
\end{table}

\cref{tab:bundle-ablation} shows that the generator-wise multicut model alone gives at most a moderate speedup over BLM, between $0.99\times$ and $1.33\times$, and that PC-BLM and PC-BPLM give the largest speedups on both datasets.
All methods reach the target on all instances of both datasets, so the comparison depends only on runtime.
On the Belgian dataset, PC-BPLM performs best, followed by PC-BLM and MC-BPLM, while MC-BLM matches BLM.
On the Californian dataset, PC-BPLM and PC-BLM give the largest speedups, $2.35\times$ and $1.90\times$, ahead of MC-BPLM and MC-BLM.
DLM is slower than BLM on both datasets.
The threshold here is based on the dataset's mean minimum side payment, while the main benchmark uses an instance-specific reference optimum, so speedups under the two protocols are not directly comparable.

These results support excluding DLM, since it is slower than BLM on both datasets.
They also support retaining the variable-metric variants PC-BLM and PC-BPLM for the main benchmark.
MC-BLM and MC-BPLM are not retained, since their gains are smaller than those of the variable-metric variants.

\subsection{Main benchmark}
\label{subsec:benchmark}
We now compare the retained bundle methods (PC-BLM and PC-BPLM) with the subgradient methods, the recent parameter-free methods, the smoothing based method, and the primal column generation reference. We point out the fact that the column generation method is also implemented with a multi-cut polyhedral master problem for a fair comparison with the bundle methods.
All methods are evaluated on the 24 instances described above with a time limit of $15$ minutes.
The results are reported in \cref{fig:summary-plots}.
For each method, the plots show the evolution over time of the mean excess side payment over all instances in each dataset.
\begin{figure*}[htbp]
	\centering
	\includegraphics[width=0.90\textwidth]{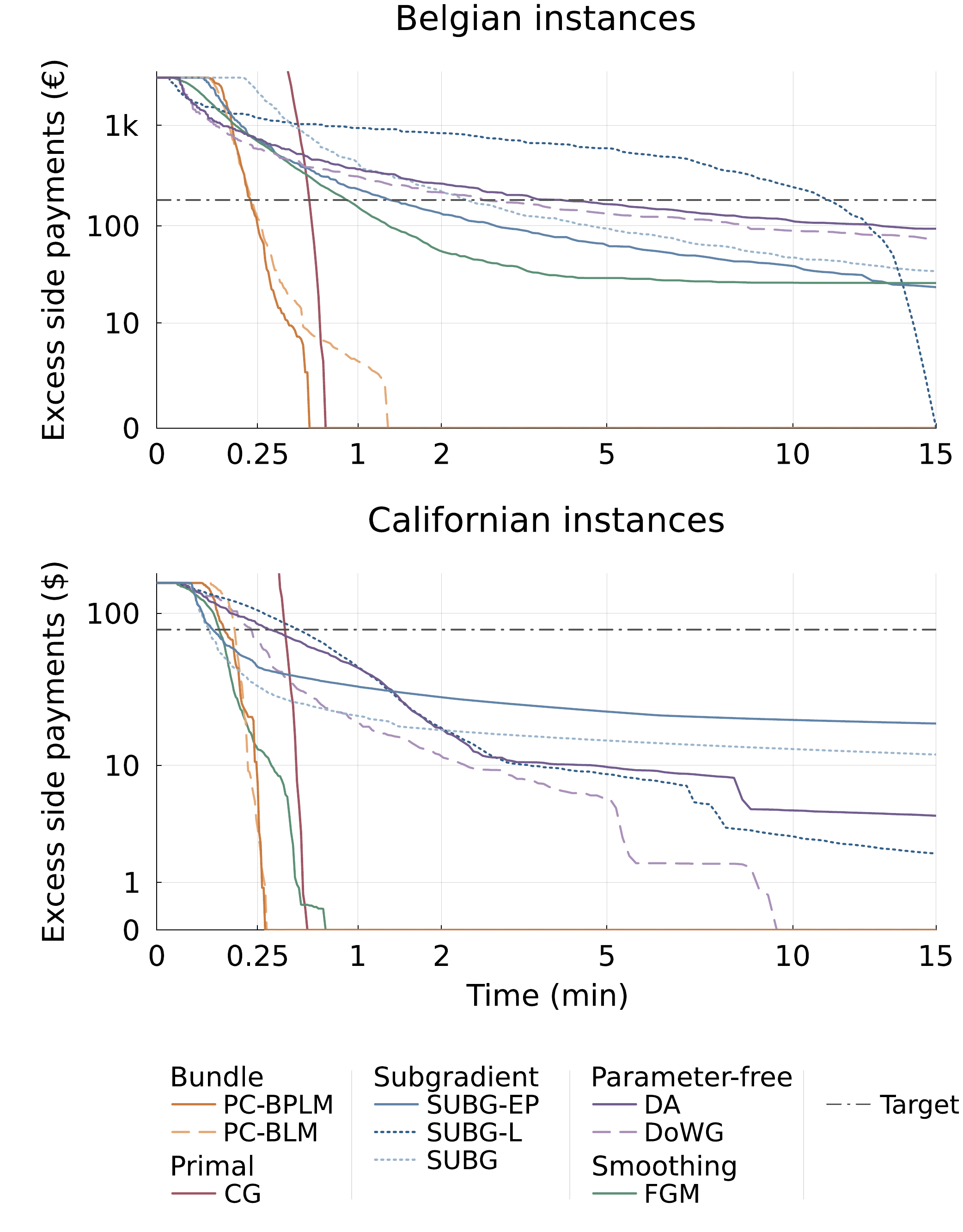}%
		\caption{Mean excess side payments over time for each dataset. For a given method and a given instance, the excess side payment at time $t$ is the current side payment minus the minimum side payment. The plotted curves are obtained by evaluating these per-instance trajectories on a common time grid and then averaging them arithmetically over all instances of the dataset. Time is shown on a square-root scale and excess side payments on a $\log(1+y)$ scale, and excess side payments below $10^{-6}$ times the minimum side payment are drawn as zero.
		The dash-dot target line marks the level $10^{-5}$ times the mean minimum side payment.
		The corresponding mean minimum side payment is $18.17$ million euros for the Belgian dataset and $7.85$ million dollars for the Californian dataset.}%
	\label{fig:summary-plots}%
\end{figure*}

\noindent Our main findings are summarized as follows:
	\begin{itemize}
		\item \textbf{Methods based on multicut polyhedral models compute exact CH prices.} PC-BLM, PC-BPLM and CG are the only methods to reach the optimal solution in the $15$-minute time scale on all instances of both datasets.
		\item \textbf{Bundle methods give stronger intermediate solutions than CG.} CG starts without enough columns to cover the demand, so its restricted master problem sheds load at the value of lost load and its primal value remains far above the optimum during the first iterations. CG reaches the target line after about $0.6$ minutes on the Belgian dataset and $0.4$ minutes on the Californian dataset, against about $0.2$ and $0.1$ minutes for the retained bundle methods, which start from the continuous-relaxation prices and decrease excess side payments steadily.
		\item \textbf{Subgradient methods give strong early solutions on the Californian dataset.} SUBG and SUBG-EP reach the target line in about $0.08$ minutes on the Californian dataset, slightly before the retained bundle methods, and SUBG-L in about $0.50$ minutes, while the retained bundle methods are stronger on the Belgian dataset.
		\item \textbf{FGM is the fastest method outside the bundle family on the Belgian dataset.} On the Belgian dataset, FGM reaches the target line after about $0.9$ minutes, before every subgradient and parameter-free method, and its geometric-mean time to reach the tolerance of \cref{tab:time-reach} is $36.3$ seconds, against $12.2$ and $12.9$ seconds for PC-BPLM and PC-BLM and at least $79.6$ seconds for the other methods. On the Californian dataset, it reaches the target line after about $0.1$ minutes and has the smallest geometric-mean time to reach this tolerance, $3.5$ seconds.
		\item \textbf{The parameter-free methods play a secondary role.} On the Belgian dataset, DA and DoWG end around $93$ and $72$ euros, against $26$ euros for FGM and less than $0.1$ euros for the retained bundle methods. On the Californian dataset, DA, DoWG, and FGM end within a few dollars of the optimum.
	\end{itemize}

The proposed hyperparameter selection strategy uses a single tuning pass per dataset: one for the eight Belgian instances and one for the 16 Californian instances.
To assess how this strategy affects performance, \cref{tab:error-reached-summary} reports the final excess side payments for all instances.
\begin{table}[H]
\centering
\footnotesize
\setlength{\tabcolsep}{2.8pt}
\renewcommand{\arraystretch}{0.95}
\caption{Final excess side payments after 15~min for all first-order methods.}
\label{tab:error-reached-summary}
\begin{tabular}{l|cc|ccc|cc|c}
\toprule
Dataset & \multicolumn{2}{c|}{Bundle} & \multicolumn{3}{c|}{Subgradient} & \multicolumn{2}{c|}{Parameter-free} & Smoothing \\
& PC-BPLM & PC-BLM & SUBG-EP & SUBG-L & SUBG & DA & DoWG & FGM \\
\midrule
\multirow{8}{*}{BE}
& \emph{0.3} & \emph{\textbf{0.0}} & \emph{101.2} & \emph{6.1} & \emph{81.5} & \emph{199.1} & \emph{177.7} & \emph{24.3} \\
& \textbf{0.0} & \textbf{0.0} & 20.2 & 4.1 & 31.3 & 76.7 & 58.1 & 25.7 \\
& \textbf{0.0} & 0.1 & 14.5 & 1.5 & 29.3 & 79.8 & 47.6 & 25.3 \\
& 0.1 & \textbf{0.0} & 11.2 & 4.4 & 18.6 & 67.3 & 46.3 & 20.2 \\
& 0.1 & \textbf{0.0} & 19.3 & 3.5 & 25.8 & 87.9 & 55.2 & 25.5 \\
& \textbf{0.0} & \textbf{0.0} & 24.7 & 7.5 & 34.5 & 68.2 & 68.0 & 19.6 \\
& \textbf{0.0} & 0.3 & 21.6 & 4.8 & 29.5 & 93.0 & 73.9 & 25.9 \\
& \textbf{0.0} & 0.2 & 23.2 & 6.7 & 24.3 & 71.9 & 52.9 & 40.3 \\
\midrule
\multirow{16}{*}{CA}
& \emph{\textbf{0.0}} & \emph{\textbf{0.0}} & \emph{0.5} & \emph{0.4} & \emph{0.6} & \emph{1.3} & \emph{2.8} & \emph{4.4} \\
& 0.4 & \textbf{0.0} & 0.5 & 0.4 & 0.6 & 1.3 & 2.8 & 4.4 \\
& 0.2 & \textbf{0.0} & 0.6 & 0.4 & 0.6 & 1.3 & 2.7 & 4.2 \\
& 0.2 & \textbf{0.0} & 0.5 & 0.4 & 0.6 & 1.4 & 2.8 & 4.4 \\
& 0.1 & \textbf{0.0} & 18.6 & 7.3 & 20.7 & 16.6 & 3.0 & 3.1 \\
& \textbf{0.0} & \textbf{0.0} & 18.0 & 6.6 & 20.7 & 16.6 & 3.0 & 2.7 \\
& \textbf{0.0} & 0.1 & 18.7 & 7.5 & 20.7 & 16.9 & 3.4 & 3.1 \\
& \textbf{0.0} & 0.1 & 18.4 & 7.5 & 20.6 & 17.0 & 3.3 & 3.1 \\
& 0.1 & \textbf{0.0} & 0.1 & 0.1 & 3.3 & \textbf{0.0} & 0.2 & 2.8 \\
& \textbf{0.0} & 0.1 & 0.1 & \textbf{0.0} & 3.3 & 0.1 & 0.3 & 3.4 \\
& 0.1 & \textbf{0.0} & 0.1 & 0.1 & 3.5 & 0.1 & 0.4 & 2.7 \\
& \textbf{0.0} & \textbf{0.0} & 0.1 & 0.1 & 3.4 & 0.1 & 0.6 & 3.3 \\
& \textbf{0.0} & 0.3 & 57.8 & 5.7 & 26.1 & 6.6 & 7.1 & 10.0 \\
& \textbf{0.0} & 0.1 & 57.6 & 7.2 & 26.4 & 6.7 & 7.1 & 8.5 \\
& 0.1 & \textbf{0.0} & 58.1 & 6.2 & 26.3 & 6.9 & 7.4 & 7.3 \\
& \textbf{0.0} & 0.3 & 58.2 & 6.3 & 26.0 & 7.0 & 7.8 & 9.2 \\
\bottomrule
\end{tabular}
\vspace{2mm}
\parbox{0.98\columnwidth}{\footnotesize Values are reported in local currency units. Results for tuning instances are italicized.}
\end{table}

These results suggest that a single hyperparameter selection per dataset is adequate for the retained bundle methods, whose final excess side payments stay below $0.5$ on every instance. For the other methods, final excess side payments vary more across instances, for example from $0.1$ to $58.2$ for SUBG-EP on the Californian dataset, so these results alone do not establish the robustness of their selected hyperparameters.

\begin{table}[H]
\centering
\footnotesize
\setlength{\tabcolsep}{2.8pt}
\renewcommand{\arraystretch}{0.95}
\caption{Time to reach an excess side payment threshold equal to $10^{-5}$ times the optimal side payments.}
\label{tab:time-reach}
\begin{tabular}{l|cc|ccc|cc|c}
\toprule
Dataset & \multicolumn{2}{c|}{Bundle} & \multicolumn{3}{c|}{Subgradient} & \multicolumn{2}{c|}{Parameter-free} & Smoothing \\
& PC-BPLM & PC-BLM & SUBG-EP & SUBG-L & SUBG & DA & DoWG & FGM \\
\midrule
\multirow{8}{*}{BE}
& \emph{\textbf{14.8}} & \emph{16.7} & \emph{171.6} & \emph{688.9} & \emph{236.4} & \emph{595.9} & \emph{508.6} & \emph{110.6} \\
& 11.0 & \textbf{10.9} & 103.8 & 680.8 & 168.1 & 245.1 & 200.8 & 57.1 \\
& \textbf{11.6} & 12.0 & 52.2 & 611.5 & 121.0 & 104.3 & 57.0 & 18.7 \\
& \textbf{10.8} & 15.6 & 76.5 & 702.9 & 143.8 & 156.7 & 153.2 & 25.6 \\
& 16.6 & \textbf{10.9} & 63.2 & 702.9 & 116.8 & 201.9 & 90.8 & 23.6 \\
& \textbf{11.4} & 15.9 & 106.3 & 730.3 & 170.2 & 340.9 & 314.3 & 57.4 \\
& \textbf{12.3} & \textbf{12.3} & 67.5 & 568.2 & 108.6 & 154.2 & 79.1 & 31.8 \\
& 10.5 & \textbf{10.4} & 49.9 & 619.7 & 96.9 & 124.1 & 69.5 & 23.3 \\
GM & \textbf{12.2} & 12.9 & 79.6 & 661.0 & 139.6 & 205.2 & 139.1 & 36.3 \\
\midrule
\multirow{16}{*}{CA}
& \emph{3.5} & \emph{5.0} & \emph{\textbf{2.0}} & \emph{11.2} & \emph{2.4} & \emph{5.2} & \emph{5.8} & \emph{3.5} \\
& 5.0 & 7.9 & 2.1 & 8.7 & \textbf{2.0} & 5.5 & 5.9 & 3.8 \\
& 5.0 & 6.1 & \textbf{2.1} & 8.3 & 2.6 & 5.1 & 5.5 & 3.3 \\
& 3.3 & 4.4 & \textbf{2.2} & 11.0 & 2.6 & 5.6 & 6.0 & 3.7 \\
& 6.8 & 9.5 & 2.9 & 4.8 & 4.4 & 4.8 & 5.0 & \textbf{2.3} \\
& 6.1 & 7.3 & 2.1 & 3.9 & 4.5 & 4.9 & 5.1 & \textbf{1.9} \\
& 7.6 & 6.8 & 3.0 & 5.1 & 4.5 & 5.1 & 5.5 & \textbf{2.3} \\
& 4.9 & 7.4 & 2.2 & 3.8 & 4.5 & 5.2 & 5.3 & \textbf{2.0} \\
& 6.5 & 8.6 & 2.4 & 4.8 & 2.7 & \textbf{2.3} & 3.1 & \textbf{2.3} \\
& 4.2 & 11.2 & \textbf{2.1} & 4.3 & 2.3 & 2.5 & 3.1 & 2.3 \\
& 4.6 & 4.7 & \textbf{2.3} & 4.2 & 2.7 & 2.4 & 2.9 & 2.6 \\
& 4.1 & 6.5 & \textbf{2.0} & 5.7 & 3.0 & 2.6 & 3.2 & 2.4 \\
& \textbf{5.5} & 9.6 & 33.0 & 71.6 & 6.9 & 70.2 & 18.9 & 8.5 \\
& 10.4 & 12.1 & 32.7 & 65.1 & \textbf{6.7} & 68.5 & 19.2 & 8.0 \\
& 9.3 & 9.0 & 34.4 & 74.0 & \textbf{7.3} & 73.3 & 19.3 & 8.0 \\
& 14.8 & 9.9 & 34.6 & 67.0 & \textbf{7.0} & 70.0 & 19.9 & 7.4 \\
GM & 5.8 & 7.6 & 4.4 & 10.9 & 3.7 & 8.2 & 6.5 & \textbf{3.5} \\
\bottomrule
\end{tabular}
\vspace{2mm}
\parbox{0.98\columnwidth}{\footnotesize Results for tuning instances are italicized. GM denotes the geometric mean over the instances of the dataset.}
\end{table}

\cref{tab:time-reach} shows the time each method required to reach an excess side payment threshold equal to $10^{-5}$ times the optimal side payments and reports the geometric mean time of each method.
All displayed methods reached the threshold on all instances.
In the Belgian dataset, PC-BPLM and PC-BLM were the fastest, with geometric means of $12.2$ s and $12.9$ s, against $36.3$ s for FGM.
In the Californian dataset, geometric means were $3.5$ s for FGM, $3.7$ s for SUBG, $4.4$ s for SUBG-EP, $5.8$ s for PC-BPLM and $7.6$ s for PC-BLM.

\myparagraph{Bundle method stopping criteria}
Bundle methods track an explicit gap $\Delta = UB - LB$, enabling a stopping
rule that does not require knowledge of $\cL_*$. The threshold used in
\cref{tab:time-reach} is an excess side payment of $10^{-5}$ times the optimal
side payments; on the Belgian dataset, where mean optimal side payments are
approximately $1.82 \times 10^7$, this corresponds to a gap of about $182$ in local currency units, which is within standard market clearing tolerances.
The exact level at which a remaining gap becomes economically negligible is a market-design decision for stakeholders.
For PC-BLM, this threshold was reached on all instances before the
15-minute budget, with mean crossing times of $18.0$ s on BE and $8.5$ s on
CA. For PC-BPLM, mean crossing times are $15.7$ s on BE and $6.0$ s on CA.
For non-bundle methods (SUBG, SUBG-EP, SUBG-L, FGM, DA, DoWG), no gap estimate is available and
stopping is time-based; adaptive stopping rules remain an interesting direction
for future work.
\subsection{Computational diagnostics}
We complement the main benchmark with computational diagnostics on the per-iteration cost of the methods.

\myparagraph{Oracle statistics}
To contextualize the per-iteration cost of all methods, we measure the time to solve
each single-generator MILP subproblem. A full oracle evaluation solves one such
subproblem per generator. Results are summarized in \cref{tab:oracle-stats}.

\begin{table}[H]
\centering
\caption{%
MILP subproblem statistics per dataset.
Mean number of variables and constraints in a single-generator MILP subproblem,
and mean and standard deviation of its solve time.
}
\label{tab:oracle-stats}
\begin{tabular}{@{}l c c c c@{}}
\toprule
Dataset & \# vars & \# constrs & mean MILP (s) & std MILP (s) \\
\midrule
BE & $488$ & $748$ & $0.0071$ & $0.0074$ \\
CA & $248$ & $393$ & $0.0021$ & $0.0006$ \\
\bottomrule
\end{tabular}
\end{table}

\subsection{Sensitivity and robustness analyses}
We next test the robustness of the conclusions with respect to numerical implementation choices and hyperparameter misspecification.

\myparagraph{Sensitivity to the smoothing parameter}
We ran FGM for 5~min on all 24 instances with $\varsigma \in \{10^{-9}, 10^{-8}, 10^{-7}\}$
and the step size $c \eta^* \varsigma / 10^{-8}$, where $\eta^*$ is the step size tuned for
$\varsigma = 10^{-8}$ and $c \in \{1/3, 1, 3\}$.
For each instance and each $c$, let $\hat{\pi}_\varsigma$ be the point returned by FGM with smoothing parameter $\varsigma$, and let the relative spread be
\begin{equation*}
\delta_c = \frac{\max_\varsigma \bar{\cL}(\hat{\pi}_\varsigma) - \min_\varsigma \bar{\cL}(\hat{\pi}_\varsigma)}{|\cL_*|},
\end{equation*}
where the maximum and the minimum are taken over the three values of $\varsigma$.
\cref{tab:alpha_vs_varsigma_fmin_new} reports the largest value of $\delta_c$ over each dataset.
\begin{table}[H]
\centering
\caption{Largest value of $\delta_c$ (\%) over each dataset, for each step-size factor $c$.}
\label{tab:alpha_vs_varsigma_fmin_new}
\begin{tabular}{@{} l c c c @{}}
\toprule
& \multicolumn{3}{c}{$c$} \\
\cmidrule(l){2-4}
Dataset & $1/3$ & $1$ & $3$ \\
\midrule
BE   & $1.5 \times 10^{-3}$ & $1.6 \times 10^{-3}$ & $5.4 \times 10^{-3}$ \\
CA   & $2.6 \times 10^{-4}$ & $8.8 \times 10^{-4}$ & $2.6 \times 10^{-3}$ \\
\bottomrule
\end{tabular}
\end{table}

A larger smoothing parameter degrades the convergence of FGM: with $\varsigma = 10^{-7}$, the mean excess side payment after 5~min is 4 to 11 times larger than with $\varsigma = 10^{-8}$, for every $c$ and on both datasets.
A smaller smoothing parameter, on the other hand, is not necessarily preferable, as it worsens the conditioning of the smoothed problem, whose gradient is $(1 + |\cG|)/\varsigma$-Lipschitz.
Indeed, although $\varsigma = 10^{-9}$ improves convergence on the Californian dataset, it degrades it on the Belgian dataset for some values of $c$.
With $\varsigma = 10^{-8}$ and $c = 1$, FGM reaches the target accuracy on every instance within 5~min.

\myparagraph{Parameter sensitivity}
To assess sensitivity to parameter misspecification, we vary each method's tunable
parameter over the coarse grid used for tuning and measure the ratio between the
worst and best mean excess side payments across the grid. Excess side payments below $10^{-6}$
times the optimal value, the precision of the primal solver, are counted as $10^{-6}$ times the optimal
value. A ratio close to 1 indicates
that performance is stable across parameter values, while a large ratio indicates high
sensitivity. We report one representative method per family: SUBG-EP for subgradient-type
methods, PC-BPLM for bundle methods, DA for parameter-free methods, and FGM for smoothing.

\begin{table}[ht]
\centering
\small
\caption{Sensitivity to parameter misspecification for one representative method per family.}
\label{tab:parameter-robustness}
\begin{tabular}{llcc}
\toprule
Family & Method & BE & CA \\
\midrule
Subgradient-type & SUBG-EP & 19.41 & 16.41 \\
Bundle & PC-BPLM & 1.00 & 1.00 \\
Parameter-free & DA & 2.88 & 12.10 \\
Smoothing & FGM & 64.16 & 20.53 \\
\bottomrule
\end{tabular}
\vspace{2mm}
\parbox{0.98\columnwidth}{\footnotesize Entries report the ratio between the worst and the best
mean excess side payments over the tested parameter grid; lower values indicate greater
robustness. Ratios are computed from $5$-minute runs, averaged over three Belgian and four
Californian instances.}
\end{table}

\cref{tab:parameter-robustness} indicates that bundle methods are the least sensitive to
parameter misspecification, followed by DA, whereas SUBG-EP and especially FGM exhibit
higher sensitivity, most notably on the Belgian dataset.
\section{Conclusion} \label{sec:conclusion}
In this paper, we proposed to compute CH prices using a large set of first-order methods
applied to the dual formulation. All proposed methods are novel in the context of CH
pricing, except the bundle level method (BLM), its multi-cut variant (MC-BLM) and the subgradient method (SUBG). We also considered a smooth approximation
of the dual CH pricing problem solved by FGM, an accelerated first-order method.
To conduct a computational study of all presented algorithms, we provide the
\texttt{ConvexHullPricing.jl} package and use it on 24 large-scale real-world instances.
Subgradient methods, although apparently neglected by the community, are particularly
competitive when equipped with good initialization and stepsize schedules, as illustrated
by SUBG and SUBG-EP for early accuracy and by SUBG-L for final accuracy. The central finding is that exploiting the separability of the
Lagrangian objective through clustered multi-cut under-approximations, combined with a
variable-metric projection, yields the strongest methods. Among the methods of the main
benchmark, the variable-metric multi-cut bundle methods (PC-BLM and PC-BPLM) clearly
dominate across both datasets: from the first 30 seconds onward, both reach lower mean
excess side payments than all the other first-order methods. On the Californian dataset,
FGM and the subgradient methods SUBG and SUBG-EP reach the target tolerance slightly
earlier. A systematic ablation (\cref{tab:bundle-ablation}) shows that generator-wise
multi-cut models alone yield at most a moderate speedup, up to $1.33\times$, whereas PC-BLM and PC-BPLM are $1.25$ to $2.35$
times faster than BLM. We therefore recommend the proximal level method with clustered multi-cut models and a variable-metric projection (PC-BPLM). Primal methods, although not the main object of study of this paper, are of
independent interest; our results indicate that Column Generation is competitive,
though it does not match the anytime performance of the proposed bundle methods. Further
analysis of decomposition methods applied to extended formulations
\cite{sadykov2013column,yu2020extended,paquetcomparison} is a natural direction for
future work.
\section*{Acknowledgements}
This work has been supported by the ITN-ETN project TraDE-OPT funded by the EU’s Horizon 2020 research and innovation programme under the Marie Sklodowska-Curie grant agreement No. 861137. The first author was also partially funded by the UCLouvain \emph{Fonds spéciaux de Recherche} (FSR). The authors are grateful to the handling editors and the anonymous referees, whose comments helped improve this paper.

\bibliographystyle{elsarticle-num}
\bibliography{references}

@article{madani2017mip,
  title = {A MIP framework for non-convex uniform price day-ahead electricity auctions},
  author = {Madani, Mehdi and Van Vyve, Mathieu},
  journal = {EURO Journal on Computational Optimization},
  volume = {5},
  number = {1},
  pages = {263--284},
  year = {2017},
  publisher = {Springer}
}

@article{Geoffrion_1974,
  author = {Geoffrion, Arthur M.},
  title = {Lagrangian relaxation for integer programming},
  journal = {Mathematical Programming Study},
  year = {1974},
  pages = {82--114}
}

@article{fernandez2015bilevel,
  title = {Bilevel programming for price-based electricity auctions: A revenue-constrained case},
  author = {Fern{\'a}ndez-Blanco, Ricardo and Arroyo, Jos{\'e} M and Alguacil, Natalia},
  journal = {EURO Journal on Computational Optimization},
  volume = {3},
  number = {3},
  pages = {163--195},
  year = {2015},
  publisher = {Springer}
}

@article{andrianesis2021computation,
  title = {Computation of convex hull prices in electricity markets with non-convexities using {D}antzig-{W}olfe decomposition},
  author = {Andrianesis, Panagiotis and Bertsimas, Dimitris and Caramanis, Michael C and Hogan, William W},
  year = 2021,
  journal = {IEEE Transactions on Power Systems},
  publisher = {IEEE},
  volume = 37,
  number = 4,
  pages = {2578--2589}
}

@article{beck2009fast,
  title = {A fast iterative shrinkage-thresholding algorithm for linear inverse problems},
  author = {Beck, Amir and Teboulle, Marc},
  journal = {SIAM journal on imaging sciences},
  volume = {2},
  number = {1},
  pages = {183--202},
  year = {2009},
  publisher = {SIAM}
}

@article{boyd2003subgradient,
  title = {Subgradient methods},
  author = {Boyd, Stephen and Xiao, Lin and Mutapcic, Almir},
  year = 2003,
  journal = {lecture notes of EE392o, Stanford University, Autumn Quarter},
  volume = 2004,
  pages = {2004--2005}
}

@article{briant2008comparison,
  title = {Comparison of bundle and classical column generation},
  author = {Briant, Olivier and Lemar{\'e}chal, Claude and Meurdesoif, Ph and Michel, Sophie and Perrot, Nancy and Vanderbeck, Fran{\c{c}}ois},
  year = 2008,
  journal = {Mathematical programming},
  publisher = {Springer},
  volume = 113,
  pages = {299--344}
}

@article{BYERS2023351,
  title = {Long-run optimal pricing in electricity markets with non-convex costs},
  journal = {European Journal of Operational Research},
  volume = {307},
  number = {1},
  pages = {351-363},
  year = {2023},
  author = {Conleigh Byers and Gabriela Hug}
}

@book{conforti2014integer,
  title = {Integer programming models},
  author = {Conforti, Michele and Cornu{\'e}jols, G{\'e}rard and Zambelli, Giacomo and Conforti, Michele and Cornu{\'e}jols, G{\'e}rard and Zambelli, Giacomo},
  year = {2014},
  publisher = {Springer}
}

@inproceedings{cutkosky2017online,
  title = {Online learning without prior information},
  author = {Cutkosky, Ashok and Boahen, Kwabena},
  booktitle = {Conference on learning theory},
  pages = {643--677},
  year = {2017},
  organization = {PMLR}
}

@inproceedings{defazio2023dadapt,
  title = {Learning-rate-free learning by {D}-{A}daptation},
  author = {Defazio, Aaron and Mishchenko, Konstantin},
  booktitle = {International Conference on Machine Learning},
  pages = {7449--7479},
  year = {2023},
  organization = {PMLR}
}

@article{frangioni2020standard,
  title = {Standard bundle methods: untrusted models and duality},
  author = {Frangioni, Antonio},
  year = 2020,
  journal = {Numerical nonsmooth optimization: state of the art algorithms},
  publisher = {Springer},
  pages = {61--116}
}

@article{gribik2007market,
  title = {Market-clearing electricity prices and energy uplift},
  author = {Gribik, Paul R and Hogan, William W and Pope, Susan L and others},
  year = 2007,
  journal = {Cambridge, MA},
  pages = {1--46}
}

@misc{interconnection2017proposed,
  title = {Proposed enhancements to energy price formation},
  author = {{Interconnection PJM}},
  year = {2017},
  publisher = {Technical report}
}

@article{hua2016convex,
  title = {A convex primal formulation for convex hull pricing},
  author = {Hua, Bowen and Baldick, Ross},
  year = 2016,
  journal = {IEEE Transactions on Power Systems},
  publisher = {IEEE},
  volume = 32,
  number = 5,
  pages = {3814--3823}
}

@inproceedings{ito2013convex,
  title = {Convex hull pricing for demand response in electricity markets},
  author = {Ito, Naoki and Takeda, Akiko and Namerikawa, Toru},
  booktitle = {2013 IEEE International Conference on Smart Grid Communications (SmartGridComm)},
  pages = {151--156},
  year = {2013},
  organization = {IEEE}
}

@article{khaled2023dowg,
  title = {Do{WG} {U}nleashed: {A}n {E}fficient {U}niversal {P}arameter-{F}ree {G}radient {D}escent {M}ethod},
  author = {Khaled, Ahmed and Mishchenko, Konstantin and Jin, Chi},
  journal = {Advances in Neural Information Processing Systems},
  volume = {36},
  pages = {6748--6769},
  year = {2023}
}

@article{knueven2020mixed,
  title = {On mixed-integer programming formulations for the unit commitment problem},
  author = {Knueven, Bernard and Ostrowski, James and Watson, Jean-Paul},
  year = 2020,
  journal = {INFORMS Journal on Computing},
  publisher = {INFORMS},
  volume = 32,
  number = 4,
  pages = {857--876}
}

@article{knueven2022computationally,
  title = {A computationally efficient algorithm for computing convex hull prices},
  author = {Knueven, Bernard and Ostrowski, James and Castillo, Anya and Watson, Jean-Paul},
  year = 2022,
  journal = {Computers \& Industrial Engineering},
  publisher = {Elsevier},
  volume = 163,
  pages = 107806
}

@article{Lemaréchal1995new,
  title = {New variants of bundle methods},
  author = {Lemar{\'e}chal, Claude and Nemirovskii, Arkadii and Nesterov, Yurii},
  year = 1995,
  journal = {Mathematical Programming},
  volume = 69,
  number = 1,
  pages = {111--147},
  doi = {10.1007/BF01585555},
  url = {https://doi.org/10.1007/BF01585555}
}

@misc{li2023simple,
  title = {A simple uniformly optimal method without line search for convex optimization},
  author = {Tianjiao Li and Guanghui Lan},
  year = 2023,
  eprint = {2310.10082}
}

@misc{madani2018convex,
  title = {Convex Hull, {IP} and {E}uropean {E}lectricity {P}ricing in a {E}uropean {P}ower {E}xchanges setting with efficient computation of {C}onvex {H}ull {P}rices},
  author = {Madani, Mehdi and Ruiz, Carlos and Siddiqui, Sauleh and Van Vyve, Mathieu},
  year = 2018,
  publisher = {arXiv}
}

@article{devolder2014first,
  title = {First-order methods of smooth convex optimization with inexact oracle},
  author = {Devolder, Olivier and Glineur, Fran{\c{c}}ois and Nesterov, Yurii},
  year = {2014},
  journal = {Mathematical Programming},
  publisher = {Springer},
  volume = {146},
  pages = {37--75}
}

@article{nesterov2005smooth,
  title = {Smooth minimization of non-smooth functions},
  author = {Nesterov, Yu},
  year = 2005,
  journal = {Mathematical programming},
  publisher = {Springer},
  volume = 103,
  pages = {127--152}
}

@article{nesterov2015universal,
  title = {Universal gradient methods for convex optimization problems},
  author = {Nesterov, Yu},
  year = 2015,
  journal = {Mathematical Programming},
  publisher = {Springer},
  volume = 152,
  number = {1-2},
  pages = {381--404}
}

@book{nesterovIntroductoryLecturesConvex2004,
  title = {Introductory {{Lectures}} on {{Convex Optimization}}},
  author = {Nesterov, Yurii},
  year = {2004},
  series = {Applied {{Optimization}}},
  publisher = {Springer US},
  doi = {10.1007/978-1-4419-8853-9},
  url = {http://link.springer.com/10.1007/978-1-4419-8853-9}
}

@article{orabona2016coin,
  title = {Coin betting and parameter-free online learning},
  author = {Orabona, Francesco and P{\'a}l, D{\'a}vid},
  journal = {Advances in Neural Information Processing Systems},
  volume = {29},
  year = {2016}
}

@mastersthesis{paquetcomparison,
  title = {Comparison between row generation, column generation and column-and-row generation for computing convex hull prices in day-ahead electricity markets},
  author = {Paquet, Maxime},
  year = {2021},
  url = {http://hdl.handle.net/2078.1/thesis:30590},
  school = {UCL - Ecole polytechnique de Louvain}
}

@article{polyak1963gradient,
  title = {Gradient methods for the minimisation of functionals},
  author = {Polyak, Boris T},
  year = 1963,
  journal = {USSR Computational Mathematics and Mathematical Physics},
  publisher = {Elsevier},
  volume = 3,
  number = 4,
  pages = {864--878}
}

@phdthesis{ring1995dispatch,
  author = {Ring, Brendan J},
  title = {Dispatch based pricing in decentralised power systems},
  school = {University of Canterbury. Management},
  year = {1995}
}

@article{roulet2019sharpness,
  title = {Sharpness, restart and acceleration},
  author = {Roulet, Vincent and d'Aspremont, Alexandre},
  journal = {Advances in Neural Information Processing Systems},
  volume = {30},
  year = {2017}
}

@article{sadykov2013column,
  title = {Column generation for extended formulations},
  author = {Sadykov, Ruslan and Vanderbeck, Fran{\c{c}}ois},
  year = 2013,
  journal = {EURO Journal on Computational Optimization},
  volume = 1,
  number = 1,
  pages = {81--115},
  doi = {10.1007/s13675-013-0009-9},
  url = {https://doi.org/10.1007/s13675-013-0009-9}
}

@article{scarf1994allocation,
  title = {The allocation of resources in the presence of indivisibilities},
  author = {Scarf, Herbert E},
  year = 1994,
  journal = {Journal of Economic Perspectives},
  publisher = {American Economic Association},
  volume = 8,
  number = 4,
  pages = {111--128}
}

@article{stevens2022application,
  title = {Application of the Level Method for Computing Locational Convex Hull Prices},
  author = {Stevens, Nicolas and Papavasiliou, Anthony},
  year = 2022,
  journal = {IEEE Transactions on Power Systems},
  publisher = {IEEE},
  volume = 37,
  number = 5,
  pages = {3958--3968}
}

@article{stevens4632401some,
  title = {On some advantages of convex hull pricing for the European electricity auction},
  journal = {Energy Economics},
  volume = {134},
  pages = {107542},
  year = {2024},
  doi = {https://doi.org/10.1016/j.eneco.2024.107542},
  url = {https://www.sciencedirect.com/science/article/pii/S0140988324002500},
  author = {Nicolas Stevens and Anthony Papavasiliou and Yves Smeers}
}

@article{yu2020extended,
  title = {An extended integral unit commitment formulation and an iterative algorithm for convex hull pricing},
  author = {Yu, Yanan and Guan, Yongpei and Chen, Yonghong},
  year = 2020,
  journal = {IEEE Transactions on Power Systems},
  publisher = {IEEE},
  volume = 35,
  number = 6,
  pages = {4335--4346}
}

@misc{zamani2023exact,
  title = {Exact convergence rate of the last iterate in subgradient methods},
  author = {Moslem Zamani and François Glineur},
  year = {2023},
  eprint = {2307.11134}
}

@article{zhang2025stabilised_ejor,
  author = {Zhang, Huajie and others},
  title = {Integrated investment, retrofit and abandonment energy system planning
             with a stabilised adaptive {Benders} decomposition},
  journal = {European Journal of Operational Research},
  volume = {325},
  pages = {261--280},
  year = {2025},
  doi = {10.1016/j.ejor.2025.04.005}
}

@article{pecci2025regularized,
  author = {Pecci, Filippo and Jenkins, Jesse D.},
  title = {Regularized {Benders} Decomposition for High Performance Capacity
             Expansion Models},
  journal = {{IEEE} Transactions on Power Systems},
  volume = {40},
  number = {4},
  pages = {3105--3116},
  year = {2025},
  doi = {10.1109/TPWRS.2025.3526413}
}

@article{zhang2024stabilised_cor,
  author = {Zhang, Huajie and others},
  title = {A stabilised {Benders} decomposition with adaptive oracles for
             large-scale stochastic programming},
  journal = {Computers \& Operations Research},
  volume = {167},
  pages = {106665},
  year = {2024},
  doi = {10.1016/j.cor.2024.106665}
}

\end{document}